\documentclass{article}
\usepackage{amsfonts}
\usepackage{amsmath,amssymb,amsthm,color}
\usepackage[utf8]{inputenc}
\usepackage{amsopn}
\usepackage{latexsym}
\usepackage{float}
\usepackage{url}
\usepackage{multicol}
\usepackage{hyperref}

\makeindex
\theoremstyle{plain}
\newtheorem{teo}{Theorem}[section]

\newtheorem{pps}[teo]{Proposition}
\newtheorem{cor}[teo]{Corollary}
\newtheorem{lem}[teo]{Lemma}
\newtheorem*{dfn}{Definition}
\theoremstyle{definition}
\newtheorem{ex}[teo]{Example}
\newtheorem*{remark}{Remark}

\DeclareMathOperator{\Ad}{Ad}
\DeclareMathOperator{\ad}{ad}
\DeclareMathOperator{\Nij}{Nij}
\DeclareMathOperator{\tr}{tr}
\DeclareMathOperator{\id}{id}
\begin{document}

\title{Invariant Generalized Complex Structures on Flag Manifolds}
\author{Carlos A. B. Varea\thanks{This work was supported by FAPESP grant 2016/07029-2. CV also was financed in party by the Coordena\c{c}\~ao de Aperfei\c{c}oamento de Pessoal de Nível Superior - Brasil (CAPES) - Finance Code 001.} \and Luiz A. B. San Martin}
\date{}

\maketitle

\begin{abstract}
Let $G$ be a complex semi-simple Lie group and form its maximal flag
manifold $\mathbb{F}=G/P=U/T$ where $P$ is a minimal parabolic subgroup, $U$
a compact real form and $T=U\cap P$ a maximal torus of $U$. The aim of this
paper is to study invariant generalized complex structures on $\mathbb{F}$.
We describe the invariant generalized almost complex structures on $\mathbb{F%
}$ and classify which one is integrable. The problem reduces to the study of
invariant $4$-dimensional generalized almost complex structures restricted
to each root space, and for integrability we analyse the Nijenhuis operator
for a triple of roots such that its sum is zero. We also conducted a study
about twisted generalized complex structures. We define a new bracket
`twisted' by a closed $3$-form $\Omega $ and also define the Nijenhuis
operator twisted by $\Omega $. We classify the $\Omega $-integrable
generalized complex structure.
\end{abstract}
\noindent
\\%
\textit{AMS 2010 subject classification:} 14M15, 22F30, 53D18.

\noindent%
\textit{Key words and phrases:} Flag manifolds, Homogeneous space, Semi-simple Lie groups, Generalized complex structures.

\tableofcontents


\section{Introduction}

The subject matter of this paper are invariant generalized complex
structures on flag manifolds of semi-simple Lie groups. A generalized
complex structure is a differential geometric structure introduced by Hitchin \cite{hitchin} and further developed by Gualtieri \cite{Gualtieri}, with the purpose of studying complex and symplectic structures in a unique
framework. We refer to Gualtieri \cite{Gualtieri},\cite{gualtieri2} and Cavalcanti \cite{gil} 
for the foundations of the theory of generalized complex structures.

In this paper we consider the maximal flag manifolds of the complex Lie
groups. Let $\mathfrak{g}$ be a complex semi-simple Lie algebra, $G$ a
conected Lie group with Lie algebra $\mathfrak{g}$. Then its maximal flag
manifold is the homogeneous space $\mathbb{F}=G/P$ where $P$ is a Borel
subgroup of $G$ (minimal parabolic subgroup). If $U$ is a compact real form
of $G$ then $U$ acts transitively on $\mathbb{F}$ so that we have also the
homogeneous space $\mathbb{F}=U/T$ where $T=P\cap U$ is a maximal torus of $%
U $. We are concerned with $U$-invariant structures on $\mathbb{F}$ (we do not expect to have $G$-invariant structures, although there are some cases we do not analyze them carefully). 

Our approach to study invariant structures is to reduce the problem at the
origin $x_{0}$ of $\mathbb{F}=U/T=G/P$. A $U$-invariant almost generalized
complex structure on $\mathbb{F}$ is completely determined by a complex
structure $\mathcal{J}$ on the vector space $T_{x_{0}}\mathbb{F}\oplus
T_{x_{0}}^{\ast }\mathbb{F}$ which commutes with the isotropy representation
of $T$.

In this approach the first step is to determine the $T$-invariant complex
structures $\mathcal{J}$. We do this by decomposing the Lie algebra $%
\mathfrak{u}$ of $U$ as $\mathfrak{u}=\mathfrak{t}\oplus \mathfrak{m}$ where 
$\mathfrak{t}$ is the Lie algebra of $T$ (which is a Cartan subalgebra) and 
\begin{equation*}
\mathfrak{m}=\sum_{\alpha }\mathfrak{u}_{\alpha }
\end{equation*}%
is the sum of the root spaces in $\mathfrak{u}$, that is, $\mathfrak{u}%
_{\alpha }=\left( \mathfrak{g}_{\alpha }+\mathfrak{g}_{-\alpha }\right) \cap 
\mathfrak{u}$ and $\mathfrak{g}_{\alpha }$ is the root space in the complex
Lie algebra $\mathfrak{g}$. This way $T_{x_{0}}\mathbb{F}\oplus
T_{x_{0}}^{\ast }\mathbb{F}$ becomes identified with the sum of two copies
of $\mathfrak{m}$, namely $T_{x_{0}}\mathbb{F}\oplus T_{x_{0}}^{\ast }%
\mathbb{F}\approx \mathfrak{m}\oplus \mathfrak{m}^\ast$ with 
$\mathfrak{m}^{\ast }=\sum_{\alpha }\mathfrak{u}%
_{\alpha }^{\ast }$ (see below for the precise realization of the duals $%
\mathfrak{u}_{\alpha }^{\ast }$). This way we write an invariant complex
structure $\mathcal{J}$ on $\mathfrak{m}\oplus \mathfrak{m}^{\ast }$ as a
direct sum $\mathcal{J}=\oplus _{\alpha }\mathcal{J}_{\alpha }$ where each $%
\mathcal{J}_{\mathcal{\alpha }}$ is an invariant complex structure in $%
\mathfrak{u}_{\alpha }\oplus \mathfrak{u}_{\alpha }^{\ast }$. The set of
invariant complex structures in $\mathfrak{u}_{\alpha }\oplus \mathfrak{u}%
_{\alpha }^{\ast }$ is parametrized by $S^{1}$ so that the set of invariant $%
\mathcal{J}$'s is parametrized by $\left( S^{1}\right) ^{\dim \mathfrak{m}}$%
. The $S^{1}$-parametrization of the complex structures in $\mathfrak{u}%
_{\alpha }\oplus \mathfrak{u}_{\alpha }^{\ast }$ distinguishes the cases
where $\mathcal{J}_{\alpha }$ comes from a complex or a symplectic structure
in $\mathfrak{m}$.

Related to the generalized complex structures are the Dirac structures. An
invariant Dirac structure in the homogeneous space $U/T$ is determined by an
isotropic subspace $L\subset \mathfrak{m}_{\mathbb{C}}\oplus \mathfrak{m}_{%
\mathbb{C}}^{\ast }$ where the complexified spaces are given by $\mathfrak{m}%
_{\mathbb{C}}=\sum_{\alpha }\mathfrak{g}_{\alpha }$ and $\mathfrak{m}_{%
\mathbb{C}}^{\ast }=\sum_{\alpha }\mathfrak{g}_{\alpha }^{\ast }$. Again we
can write $L=\oplus _{\alpha }L_{\alpha }$ with $L_{\alpha }\subset 
\mathfrak{g}_{\alpha }\oplus \mathfrak{g}_{\alpha }^{\ast }$ so that the set
of invariant Dirac structures become parametrized by $\left( S^{1}\right)
^{\dim \mathfrak{m}}$ as well.

Once we have the algebraic description of the invariant structures we
proceed to analyze their integrability both in the twisted and nontwisted
cases. We write algebraic equations for the integrability conditions. In
solving them we get a classification of the generalized complex structures
in the maximal flag manifolds. This is done by writing down the subset $I$
of the parameter space $\left( S^{1}\right) ^{\dim \mathfrak{m}}$ such that
the generalized almost complex structure determined by $\mathcal{J}$ is
integrable if and only if $\mathcal{J}\in I$. The same method works to get
twisted integrability with respect to an invariant closed $3$-form $H$.

A previous classification of (nontwisted) integrable structures were
provided by Milburn \cite{milb} relying on the differential $2$-form that
defines the a generalized almost complex structure. Our approach on the
other hand looks at the root spaces components of the structures in the same
spirit as in the classical papers by Borel \cite{bor} and Wolf-Gray \cite%
{gw1} and \cite{gw} (see also \cite{smartinecaio}) for the description of
the invariant complex structures. 

As to the contents of the paper in Section $2$ we present a brief
introduction to generalized complex geometry and flag manifolds, which are
the basic concepts used throughout all paper. After this, we obtain a result
which allow us to reduce the study of an invariant generalized almost
complex structure on a flag manifold $\mathbb{F}$ to the study of $4$%
-dimensional invariant generalized almost complex structures. And, in the
end of the section we describe such structures.

Section $3$ is dedicated to do a classification of all invariant generalized
almost complex structure on a flag manifold, using the results obtained in
the previous section. More than that, we classify which one of such
structure is integrable or not, by means of the Nijenhuis operator.

In Section $4$, we do a study about generalized complex structures twisted
by a closed $3$-form. We classify the invariant generalized almost complex
structures which are $\Omega$-integrable, where $\Omega$ is an invariant
closed $3$-form.

\section{Generalized Complex Geometry}

In this section we introduce the basic definitions of generalized complex
geometry. For more details see Gualtieri \cite{Gualtieri}. Let $M$ be a
smooth $n$-dimensional manifold, then the sum of the tangent and cotangent
bundle $TM\oplus T^{\ast }M$ is endowed with a natural symmetric bilinear
form with signature $(n,n)$ defined by 
\begin{equation*}
\langle X+\xi ,Y+\eta \rangle =\frac{1}{2}(\xi (Y)+\eta (X)).
\end{equation*}%
Furthermore, the Courant bracket is a skew-symmetric bracket defined on
smooth sections of $TM\oplus T^{\ast }M$ by 
\begin{equation*}
\lbrack X+\xi ,Y+\eta ]=[X,Y]+\mathcal{L}_{X}\eta -\mathcal{L}_{Y}\xi -\frac{%
1}{2}d\left( i_{X}\eta -i_{Y}\xi \right) .
\end{equation*}

\begin{dfn}
A generalized almost complex structure on $M$ is a map $\mathcal{J}\colon
TM\oplus T^{\ast }M\rightarrow TM\oplus T^{\ast }M$ such that $\mathcal{J}%
^{2}=-1$ and $\mathcal{J}$ is an isometry of the bilinear form $\langle
\cdot ,\cdot \rangle $. The generalized almost complex structure $\mathcal{J}
$ is said to be integrable to a generalized complex structure when its $i$%
-eigenbundle $L\subset (TM\oplus T^{\ast }M)\otimes \mathbb{C}$ is Courant
involutive.
\end{dfn}

Note that, given $L$ a maximal isotropic sub-bundle of $TM\oplus T^{\ast }M$
(or its complexification) then $L$ is Courant involutive if and only if $\Nij%
|_{L}=0$, where $\Nij$ is the Nijenhuis operator: 
\begin{equation*}
\Nij(A,B,C)=\frac{1}{3}\left( \langle \lbrack A,B],C\rangle +\langle \lbrack
B,C],A\rangle +\langle \lbrack C,A],B\rangle \right) .
\end{equation*}

The basic examples of generalized complex structures come from the complex
and symplectic structures. If $J$ and $\omega $ are complex and symplectic
structures respectively on $M$, then 
\begin{equation*}
\mathcal{J}_{J}=\left( 
\begin{array}{cc}
-J & 0 \\ 
0 & J^{\ast }%
\end{array}%
\right) \ \mathrm{and}\ \mathcal{J}_{\omega }=\left( 
\begin{array}{cc}
0 & -\omega ^{-1} \\ 
\omega & 0%
\end{array}%
\right)
\end{equation*}%
are generalized complex structures on $M$.

\subsubsection*{Twisted generalized complex structures}

As was developed by \v{S}evera and Weinstein, in \cite{severa}, the Courant
bracket can be `twisted' by a closed $3$-form $H$. Given a $3$-form $H$,
define another bracket $[\cdot ,\cdot ]_{H}$ on $T\oplus T^{\ast }$, by 
\begin{equation*}
\lbrack X+\xi ,Y+\eta ]_{H}=[X+\xi ,Y+\eta ]+i_{Y}i_{X}H.
\end{equation*}%
Then, defining $\Nij_{H}$ using the usual formula but replacing $[\cdot
,\cdot ]$ by $[\cdot ,\cdot ]_{H}$, one gets 
\begin{equation*}
\Nij_{H}(A,B,C)=\Nij(A,B,C)+H(X,Y,Z),
\end{equation*}%
where $A=X+\xi $, $B=Y+\eta $ and $C=Z+\zeta $.

\begin{dfn}
A generalized almost complex structure $\mathcal{J}$ is said to be a twisted
generalized complex structure with respect to the closed $3$-form $H$ (or
just $H$-integrable) when its $i$-eigenbundle $L$ is involutive with respect
to the $H$-twisted Courant bracket.
\end{dfn}

Analogously to the untwisted case, one can prove that $L$ is involutive with
respect to the $H$-twisted Courant bracket if and only if $\Nij_H |_L = 0$.

\subsection{Flag manifolds}

\label{secaoflag}Let $G$ be a complex semi-simple Lie group with Lie algebra $%
\mathfrak{g}$. Given a Cartan subalgebra $\mathfrak{h}$ of $\mathfrak{g}$
denote by $\Pi $ the set of roots of the pair $(\mathfrak{g},\mathfrak{h})$,
so that 
\begin{equation*}
\mathfrak{g}=\mathfrak{h}\oplus \sum_{\alpha \in \Pi }\mathfrak{g}_{\alpha },
\end{equation*}%
where $\mathfrak{g}_{\alpha }=\{X\in \mathfrak{g}:\forall H\in \mathfrak{h}%
,[H,X]=\alpha (H)X\}$ denotes the corresponding one-dimensional root space.
The Cartan--Killing form $\langle X,Y\rangle =\tr(\ad(X)\ad(Y))$ of $%
\mathfrak{g}$ is nondegenerate on $\mathfrak{h}$. Given $\alpha \in 
\mathfrak{h}^{\ast }$ we let $H_{\alpha }$ be defined by $\alpha (\cdot
)=\langle H_{\alpha },\cdot \rangle $, and denote by $\mathfrak{h}_{\mathbb{R%
}}$ the subspace spanned over $\mathbb{R}$ by $H_{\alpha }$, $\alpha \in \Pi 
$.

We fix a Weyl basis of $\mathfrak{g}$ which amounts to giving $X_\alpha \in 
\mathfrak{g}_\alpha$ such that $\langle X_\alpha, X_{-\alpha} \rangle = 1$,
and $[X_\alpha, X_{\beta}]=m_{\alpha,\beta}X_{\alpha+\beta}$ with $%
m_{\alpha,\beta}\in \mathbb{R}$, $m_{-\alpha,-\beta}=-m_{\alpha,\beta}$ and $%
m_{\alpha,\beta}=0$ if $\alpha+\beta$ is not a root.

Let $\Pi ^{+}\subset \Pi $ be a choice of positive roots, denote by $\Sigma $
the corresponding simple system of roots and put $\mathfrak{p}=\mathfrak{h}%
\oplus \sum_{\alpha \in \Pi ^{+}}\mathfrak{g}_{\alpha }$ for the Borel
subalgebra generated by $\Pi ^{+}$. Denote by $\mathfrak{n}^{-}=\sum_{\alpha
\in \Pi ^{+}}\mathfrak{g}_{-\alpha }$ and $\mathfrak{n}^{+}=\sum_{\alpha \in
\Pi ^{+}}\mathfrak{g}_{\alpha }$. Thus a maximal (complex) flag manifold is $%
\mathbb{F}=G/P$ where $G$ is a complex semi-simple Lie group with Lie
algebra $\mathfrak{g}$ and $P$ is the normalizer of $\mathfrak{p}$ in $G$.
Then $\mathfrak{n}^{-}$ identifies to the tangent space $T_{b_{0}}\mathbb{F}$
at the origin $b_{0}$ of $\mathbb{F}$ and $\mathfrak{n}^{+}$ to the
cotangent space $T_{x_{0}}^{\ast }\mathbb{F}$ by means of the
Cartan--Killing form.

Let $\mathfrak{u}$ be a compact real form of $\mathfrak{g}$, to know, the
real subalgebra 
\begin{equation*}
\mathfrak{u}=\mathrm{span}_{\mathbb{R}}\{i\mathfrak{h}_{\mathbb{R}%
},A_{\alpha },S_{\alpha }\colon \alpha \in \Pi ^{+}\}
\end{equation*}%
where $A_{\alpha }=X_{\alpha }-X_{-\alpha }$ and $S_{\alpha }=i(X_{\alpha
}+X_{-\alpha })$. Denote by $U=\exp \mathfrak{u}$ the correspondent compact
real form of $G$. Then we can write $\mathbb{F}=U/T$ where $T=P\cap U$ is a
maximal torus of $U$.

We can also identify $\mathbb{F}=\Ad(U)H$ where $H$ is a regular element of $%
\mathfrak{t}=\mathcal{L}(T)$. In this case, we are identifying the origin $%
b_{0}$ of $\mathbb{F}$ with $H$. The tangent space of $\mathbb{F}$ at the
origin is $\sum_{\alpha \in \Pi ^{+}}\mathfrak{u}_{\alpha }$, where $%
\mathfrak{u}_{\alpha }=(\mathfrak{g}_{-\alpha }\oplus \mathfrak{g}_{\alpha
})\cap \mathfrak{u}=\mathrm{span}_{\mathbb{R}}\{A_{\alpha },S_{\alpha }\}$.

We have several ways to relate a Lie algebra to its dual. Here, we are going
to use the symplectic form of Kostant--Kirillov--Souriaux (KKS), which is
defined on a coadjoint orbit $G\cdot \alpha $, with $\alpha \in \mathfrak{g}%
^{\ast }$, by 
\begin{equation*}
\Omega _{\alpha }\left( \widetilde{X}(\alpha ),\widetilde{Y}(\alpha )\right)
=\alpha \lbrack X,Y]
\end{equation*}%
with $X,Y\in \mathfrak{g}$. In our case, we can identify the adjoint and
coadjoint orbits. In this case, we have that 
\begin{equation*}
\Omega _{b_{0}}(\widetilde{X},\widetilde{Y})=\langle H,[X,Y]\rangle
\end{equation*}%
with $X,Y\in \mathfrak{u}$.

We have $\mathfrak{u}\approx \mathfrak{u}^{\ast }$, more than that we have $%
\mathfrak{u}_{\alpha }\approx \mathfrak{u}_{\alpha }^{\ast }$. The elements
of $\mathfrak{u}_{\alpha }^{\ast }$ will be denoted by $A_{\alpha }^{\ast }$
e $S_{\alpha }^{\ast }$. Such isomorphism is given from the KKS symplectic
form, where we denote 
\begin{equation*}
X^{\ast }=\Theta _{X}(\cdot )=\frac{1}{\langle H,H_{\alpha
}\rangle }\Omega _{b_{0}}(\widetilde{X},\cdot )
\end{equation*}%
for $X^{\ast }\in \mathfrak{u}_{\alpha }^{\ast }$. For convenience, we will
write $k_{\alpha }=\frac{1}{\langle H,H_{\alpha }\rangle }$.

\subsection{Invariant generalized almost complex structures}

When $M$ is a flag manifold one can consider invariant generalized almost
complex structures. A $U$-invariant generalized almost complex structure $%
\mathcal{J}_\ast$ is completely determined by its value $\mathcal{J} \colon 
\mathfrak{n}^- \oplus \left( \mathfrak{n}^-\right) ^\ast \rightarrow 
\mathfrak{n}^- \oplus \left( \mathfrak{n}^-\right) ^\ast$ in the sum of the
tangent and cotangent space at the origin. The map $\mathcal{J}$ satisfies $%
\mathcal{J}^2 = -1$, $\mathcal{J}$ is orthogonal to the inner product $%
\langle \cdot, \cdot \rangle$ and commute with the adjoint action of $T$ on $%
\mathfrak{n}^- \oplus \left( \mathfrak{n}^-\right) ^\ast$.

We know that a generalized almost complex structure is related to an
isotropic subspace. So we are interested in invariant isotropic subspaces on 
$\mathfrak{n}^{-}\oplus \left( \mathfrak{n}^{-}\right) ^{\ast }\approx 
\mathfrak{n}^{-}\oplus \mathfrak{n}^{+}$. In this case, the natural bilinear
form is just the Cartan--Killing form.

\begin{pps}
Let $L$ be an invariant subspace of $\mathfrak{n}^{-}\oplus \mathfrak{n}^{+}$%
. Then 
\begin{equation*}
L=\sum_{\alpha >0}L\cap \left( \mathfrak{g}_{-\alpha }\oplus \mathfrak{g}%
_{\alpha }\right) .
\end{equation*}
\end{pps}

\begin{proof}
A maximal flag manifold $\mathbb{F}=U/T$ is a reductive space, then we can
look the isotropy representation as being the adjoint representation
restrict to $T$, that is, given by $\Ad|_{T}\colon T\rightarrow \mathrm{Gl}(%
\mathfrak{u})$.

Let $L\subset \mathfrak{n}^{-}\oplus \mathfrak{n}^{+}$ be an invariant
subspace for the isotropy representation, which is identified to $\Ad\oplus %
\Ad^{\ast }$ in $\mathfrak{n}^{-}\oplus \mathfrak{n}^{+}$. Since $T$ is a
torus, every $g\in T$ is written as $g=e^{iH}$ where $H\in \mathfrak{t}=%
\mathcal{L}(T)$. Let $\alpha $ be a positive root and let $X_{\alpha }\in 
\mathfrak{g}_{\alpha }$, $X_{-\alpha }\in \mathfrak{g}_{-\alpha }$, then we
have 
\begin{equation*}
\left( \Ad\oplus \Ad^{\ast }\right) (g)(X_{-\alpha }+X_{\alpha })=\Ad%
(e^{iH})X_{-\alpha }+\Ad(e^{-iH})X_{\alpha }
\end{equation*}%
because $\Ad^{\ast }(g)=\Ad(g^{-1})$. Now using the fact that $\Ad(e^{Y})=e^{%
\ad(Y)}$ follows 
\begin{eqnarray*}
\left( \Ad\oplus \Ad^{\ast }\right) (g)(X_{-\alpha }+X_{\alpha }) &=&e^{i\ad%
(H)}X_{-\alpha }+e^{-i\ad(H)}X_{\alpha } \\
&=&e^{-i\alpha (H)}X_{-\alpha }+e^{-i\alpha (H)}X_{\alpha } \\
&=&e^{-i\alpha (H)}(X_{-\alpha }+X_{\alpha })\in \mathfrak{g}_{-\alpha
}\oplus \mathfrak{g}_{\alpha }.
\end{eqnarray*}%
Since $L$ is invariant, follows that $(\Ad\oplus \Ad^{\ast })(g)(X_{-\alpha
}+X_{\alpha })\in L$. Therefore $(\Ad\oplus \Ad^{\ast })(g)(X_{-\alpha
}+X_{\alpha })\in L\cap (\mathfrak{g}_{-\alpha }\oplus \mathfrak{g}_{\alpha
})$. Thus, given $X=\sum_{\alpha >0}(X_{-\alpha }+X_{\alpha })\in L$ by
linearity we have that $X\in L\cap (\mathfrak{g}_{-\alpha }\oplus \mathfrak{g%
}_{\alpha })$.
\end{proof}

\begin{pps}
Let $L=\sum_{\alpha >0} L\cap \left(\mathfrak{g}_{-\alpha}\oplus \mathfrak{g}%
_\alpha\right)$. Then $L$ is isotropic if and only if, for each $\alpha$, 
\begin{equation*}
L_\alpha = L\cap \left(\mathfrak{g}_{-\alpha}\oplus \mathfrak{g}%
_\alpha\right)
\end{equation*}
is an isotropic subspace.
\end{pps}

\begin{proof}
If $L$ is isotropic, then we have $\langle X,Y\rangle = 0$ for all $X,Y\in L$%
. In particular, if $X,Y\in L_\alpha$ then $\langle X,Y\rangle = 0$.
Therefore $L_\alpha$ is isotropic (for each $\alpha$).

On the other hand, suppose that $L_\alpha$ is isotropic for each $\alpha$,
then $\langle X,Y\rangle = 0$ for all $X,Y\in L_\alpha$. Now, when $\alpha
\not= \beta$ we have $\langle X,Y\rangle = 0$ for all $X\in L_\alpha$ and $%
Y\in L_\beta$, because $\langle \mathfrak{g}_\alpha,\mathfrak{g}%
_\beta\rangle = 0$ unless $\beta = -\alpha$. Therefore $L$ is isotropic, as
required.
\end{proof}

Observe that $\mathfrak{g}_{-\alpha }\oplus \mathfrak{g}_{\alpha }$ is a $2$%
-dimensional complex subspace or a $4$-dimensional real subspace. The
Cartan--Killing form restricted to $\mathfrak{g}_{-\alpha }\oplus \mathfrak{g%
}_{\alpha }$ is a symmetric bilinear form with signature $2$ (over $\mathbb{R%
}$) with matrix 
\begin{equation*}
\left( 
\begin{array}{cccc}
0 & 0 & 1 & 0 \\ 
0 & 0 & 0 & 1 \\ 
1 & 0 & 0 & 0 \\ 
0 & 1 & 0 & 0%
\end{array}%
\right)
\end{equation*}%
for some basis. Moreover, note that if $L$ is maximal isotropic on $%
\mathfrak{n}^{-}\oplus \mathfrak{n}^{+}$, then each subspace $L_{\alpha }$
is maximal isotropic on $\mathfrak{g}_{-\alpha }\oplus \mathfrak{g}_{\alpha
} $, which means $\dim _{\mathbb{R}}L_{\alpha }=2$.

In this way, we are motivated to do a detailed description of the $2$%
-dimensional isotropic subspaces on $\mathfrak{g}_{-\alpha} \oplus \mathfrak{%
g}_{\alpha}$. Actually, we are going to describe all invariant generalized
almost complex structure on $\mathbb{R}^4$.

\subsubsection*{Invariant $4$-dimensional generalized almost complex
structure}

Let $Q$ be a bilinear form with matrix 
\begin{equation*}
B=\left( 
\begin{array}{cc}
0_{2\times 2} & 1_{2\times 2} \\ 
1_{2\times 2} & 0_{2\times 2}%
\end{array}
\right).
\end{equation*}
We are interested in describing the structures $J$ such that

\begin{enumerate}
\item $J$ is a complex structure: $J^2 = -1$;

\item $J$ is an isometry of $Q$, that is, $Q(Jx,Jy)=Q(x,y)$, that is, $%
J^{T}BJ=B$, which is equivalent to $J^{T}B=-BJ$.
\end{enumerate}

If $J$ satisfies $J^TB=-BJ$, we have 
\begin{equation*}
J= \left(%
\begin{array}{cc}
\alpha & \beta \\ 
\gamma & -\alpha ^T%
\end{array}
\right)
\end{equation*}
where $\beta + \beta^T = \gamma +\gamma ^T = 0$. Thus 
\begin{equation*}
J^2 = \left( 
\begin{array}{cc}
\alpha ^2 +\beta \gamma & \alpha \beta -\beta \alpha ^T \\ 
\gamma \alpha -\alpha^T \gamma & \gamma \beta +(\alpha ^2)^T%
\end{array}
\right)
\end{equation*}
and if we ask that $J^2 = -1$ then $0 = \alpha \beta - \beta \alpha ^T =
\alpha \beta + (\alpha \beta)^T$, that is, $\alpha \beta$ is a
skew-symmetric matrix. Writing 
\begin{equation*}
\alpha = \left(%
\begin{array}{cc}
a & b \\ 
c & d%
\end{array}
\right) \ \mathrm{ and } \ \beta = \left( 
\begin{array}{cc}
0 & -x \\ 
x & 0%
\end{array}
\right),
\end{equation*}
we have 
\begin{equation*}
\alpha \beta = \left( 
\begin{array}{cc}
bx & -ax \\ 
dx & -cx%
\end{array}
\right).
\end{equation*}
Therefore, if $\alpha \beta$ is skew-symmetric and $\beta \not= 0$, then $%
\alpha$ a scalar matrix, i.e., $\alpha$ is a diagonal matrix with $a=d\in 
\mathbb{R}$. Similarly using the expression $\gamma \alpha - \alpha^T \gamma
= 0$ we have the same conclusion with $\gamma$ instead $\beta$.
Reciprocally, if $\beta$ and $\gamma$ are skew-symmetric matrices and $%
\alpha $ is a scalar matrix, then $\alpha \beta$ and $\gamma \alpha$ are
skew-symmetric matrices.

The last statement ensures that $\beta=0$ if and only if $\gamma =0$. In
fact, if $\beta = 0$ and $\gamma \not=0$, then in the diagonal of $J^2$
appears that $\alpha ^2$ with $\alpha$ a real scalar matrix and it becomes
impossible to have $J^2 = -1$.

Now, if $\beta = \gamma = 0$ the only possibility is $\alpha ^2 = -1$.
Thereby, the complex structures that are isometries of $Q$ are given by:

\begin{enumerate}
\item Diagonal: 
\begin{equation*}
J = \left( 
\begin{array}{cc}
\alpha & 0 \\ 
0 & -\alpha^T%
\end{array}%
\right)
\end{equation*}
where $\alpha^2 = -1$;

\item Non-diagonal: 
\begin{equation*}
J=\left( 
\begin{array}{cc}
a\cdot \id & \beta \\ 
\gamma & -a\cdot \id%
\end{array}%
\right)
\end{equation*}%
with 
\begin{equation*}
\beta =\left( 
\begin{array}{cc}
0 & -x \\ 
x & 0%
\end{array}%
\right) \ \mathnormal{\ }\mathrm{and}\ \gamma =\left( 
\begin{array}{cc}
0 & -y \\ 
y & 0%
\end{array}%
\right)
\end{equation*}%
where $\alpha ^{2}+\beta \gamma =-1$, that is, $a^{2}-xy=-1$.
\end{enumerate}

However, we are looking for the complex structures which are invariant. In
our case, we need $J$ to be invariant by the torus action. Let $T$ be the
group of diagonal matrices in blocks 
\begin{equation*}
k=\left( 
\begin{array}{cc}
r_{t} & 0 \\ 
0 & r_{t}%
\end{array}%
\right)
\end{equation*}%
with 
\begin{equation*}
r_{t}=\left( 
\begin{array}{cc}
\cos t & -\sin t \\ 
\sin t & \cos t%
\end{array}%
\right)
\end{equation*}%
the rotation matrix. A complex structure is $T$-invariant if $kJk^{-1}=J$
for all $k\in T$. Let $J$ be an isometry of $Q$ like before. Then 
\begin{equation*}
kJk^{-1}=\left( 
\begin{array}{cc}
r_{t}\alpha r_{t}^{-1} & r_{t}\beta r_{t}^{-1} \\ 
r_{t}\gamma r_{t}^{-1} & -r_{t}\alpha ^{T}r_{t}^{-1}%
\end{array}%
\right) =\left( 
\begin{array}{cc}
r_{t}\alpha r_{t}^{-1} & \beta \\ 
\gamma & -r_{t}\alpha ^{T}r_{t}^{-1}%
\end{array}%
\right)
\end{equation*}%
because $\beta $ and $\gamma $ are skew-symmetric and therefore commutes
with the rotations.

Now, $\alpha $ is a $2\times 2$ matrix that commutes with the rotations if 
\begin{equation*}
\alpha =\left( 
\begin{array}{cc}
a & -b \\ 
b & a%
\end{array}%
\right) .
\end{equation*}%
If furthermore $\alpha ^{2}=-1$ then $a=0$ and $b=\pm 1$, so 
\begin{equation*}
\alpha =\left( 
\begin{array}{cc}
0 & -1 \\ 
1 & 0%
\end{array}%
\right) \ \mathnormal{\ }\mathrm{or}\ \alpha =\left( 
\begin{array}{cc}
0 & 1 \\ 
-1 & 0%
\end{array}%
\right) .
\end{equation*}%
Hence, the complex structures that are isometries of $Q$ and invariants by $%
T $ are given by

\begin{enumerate}
\item Diagonal: $J=\pm J_{0}$ where 
\begin{equation*}
J_{0}=\left( 
\begin{array}{cc}
\alpha & 0 \\ 
0 & \alpha%
\end{array}%
\right) \ \mathnormal{\ }\mathrm{with}\ \alpha =\left( 
\begin{array}{cc}
0 & -1 \\ 
1 & 0%
\end{array}%
\right) .
\end{equation*}

\item Non-diagonal: 
\begin{equation*}
J=\left( 
\begin{array}{cc}
a\cdot \id & \beta \\ 
\gamma & -a\cdot \id%
\end{array}%
\right)
\end{equation*}%
with 
\begin{equation*}
\beta =\left( 
\begin{array}{cc}
0 & -x \\ 
x & 0%
\end{array}%
\right) \ \mathnormal{\ }\mathrm{and}\ \gamma =\left( 
\begin{array}{cc}
0 & -y \\ 
y & 0%
\end{array}%
\right)
\end{equation*}%
such that $xy\not=0$ and $a^{2}-xy=-1$ (all structures of this type are $T$%
-invariant, because $\beta $ and $\gamma $ commute with the rotations and $%
\alpha $ is a scalar matrix).
\end{enumerate}

\section{Generalized Complex Structures on $\mathbb{F}$}

In the previous section we have seen that a generalized almost complex
structure on a flag manifold $\mathbb{F}$ can be described by a Dirac
structure, more specifically, by its $i$-eigenspace $L$ at the origin. More
than that, we can reduce the study of these structures to the restriction to 
$\mathfrak{u}_{\alpha }\oplus \mathfrak{u}_{\alpha }^{\ast }$, for each root 
$\alpha $, which reduces the study to $4$-dimensional generalized almost
complex structures. In this section, we will see which of these structures
are integrable.

\subsection{The Courant bracket and Nijenhuis operator}

The Courant bracket is defined by 
\begin{equation*}
\lbrack X+\xi ,Y+\eta ]=[X,Y]+\mathcal{L}_{X}\eta -\mathcal{L}_{Y}\xi -\frac{%
1}{2}d(i_{X}\eta -i_{Y}\xi ).
\end{equation*}%
Remembering our notation which was fixed in the section \ref{secaoflag}, the
isomorphism beetwen $\mathfrak{u}_{\alpha }$ and its dual $\mathfrak{u}%
_{\alpha }^{\ast }$ was given by 
\begin{equation*}
X^{\ast }=k_{\alpha }\Omega _{b_{0}}(\widetilde{X},\cdot ),
\end{equation*}%
for each $X^{\ast }\in \mathfrak{u}_{\alpha }^{\ast }$, where $\Omega $ is
the KKS symplectic form and $k_{\alpha }=\frac{1}{\langle H,H_{\alpha }\rangle }$. In order to describe the Courant bracket, observe
that for $X\in \mathfrak{u}$ and $Y\in \mathfrak{u}_{\alpha }$, 
\begin{eqnarray*}
\mathcal{L}_{X}{Y}^{\ast } &=&di_{X}{Y}^{\ast }+i_{X}d{Y}^{\ast } \\
&=&di_{X}\Theta _{Y}+i_{X}d\Theta _{Y} \\
&=&di_{X}\Theta _{Y}
\end{eqnarray*}%
because $\Omega _{b_{0}}$ is the height function. Thus, given $Z\in 
\mathfrak{u}$ we have 
\begin{eqnarray*}
d(i_{X}\Theta _{Y})Z &=&Z\left( k_{\alpha }\Omega _{b_{0}}(\widetilde{Y},%
\widetilde{X})\right) \\
&=&k_{\alpha }\Omega _{b_{0}}\left( [\widetilde{Z},\widetilde{Y}],\widetilde{%
X}\right) +k_{\alpha }\Omega _{b_{0}}\left( \widetilde{Y},[\widetilde{Z},%
\widetilde{X}]\right) \\
&=&k_{\alpha }\left( \langle H,[[Z,Y],X]\rangle +\langle H,[Y,[Z,X]]\rangle
\right) \\
&=&k_{\alpha }\langle H,[Z,[Y,X]]\rangle .
\end{eqnarray*}%
Thereby, given $Y\in \mathfrak{u}_{\alpha }$ and $W\in \mathfrak{u}_{\beta }$%
, the Courant bracket in $\mathfrak{u}\oplus \mathfrak{u}^{\ast }$ is 
\begin{eqnarray*}
\lbrack X+{Y}^{\ast },Z+{W}^{\ast }] &=&[X,Z]+\mathcal{L}_{X}{W}^{\ast }-%
\mathcal{L}_{Z}{Y}^{\ast }-\frac{1}{2}d\left( i_{X}{W}^{\ast }-i_{Z}{Y}%
^{\ast }\right) \\
&=&[X,Z]+d(i_{X}\Theta _{W})-d(i_{Z}\Theta _{Y})-\frac{1}{2}d(i_{X}\Theta
_{W}-i_{Z}\Theta _{Y}) \\
&=&[X,Z]+\frac{1}{2}d(i_{X}\Theta _{W})-\frac{1}{2}d(i_{Z}\Theta _{Y}) \\
&=&[X,Z]+\frac{1}{2}k_{\beta }\langle H,[\cdot ,[W,X]]\rangle -\frac{1}{2}%
k_{\alpha }\langle H,[\cdot ,[Y,Z]]\rangle .
\end{eqnarray*}

One natural question is when a generalized almost complex structure is a
generalized complex structure. One way to check this is verifying when the
Nijenhuis operator is zero, where the Nijenhuis operator is defined by 
\begin{equation*}
\Nij (A,B,C) = \frac{1}{3}(\langle [A,B],C\rangle + \langle [B,C],A\rangle+
\langle [C,A],B\rangle ).
\end{equation*}
Using the expression for the Courant bracket obtained above, for $A=A_1+{A}%
_2 ^\ast$, $B=B_1+{B}_2 ^\ast$ and $C=C_1+{C}_2 ^\ast$, where $A_2 ^\ast \in 
\mathfrak{u}^\ast _{\alpha}$, $B_2 ^\ast \in \mathfrak{u}^\ast _{\beta}$ and $C_2 ^\ast
\in \mathfrak{u}^\ast _{\gamma}$, we calculate each part of the expression of $%
\Nij
$: 
\begin{eqnarray*}
\langle [ A,B ],C\rangle = \langle [A_1,B_1]+ \frac{1}{2}k_\beta \langle H,
[\cdot ,[B_2,A_1]] \rangle - \frac{1}{2} k_\alpha\langle H, [\cdot
,[A_2,B_1]] \rangle, C_1+{C}_2 ^\ast \rangle \\
= \frac{1}{2} k_\gamma \langle H,[C_2,[A_1,B_1]]\rangle + \frac{1}{4}
k_\beta \langle H, [C_1 ,[B_2,A_1]] \rangle - \frac{1}{4} k_\alpha \langle
H, [C_1 ,[A_2,B_1]] \rangle .
\end{eqnarray*}
Analogously, 
\begin{eqnarray*}
\langle [ B,C ],A\rangle = \frac{1}{2} k_\alpha \langle
H,[A_2,[B_1,C_1]]\rangle + \frac{1}{4} k_\gamma \langle H, [A_1 ,[C_2,B_1]]
\rangle \\
- \frac{1}{4} k_\beta \langle H, [A_1 ,[B_2,C_1]] \rangle
\end{eqnarray*}
and 
\begin{eqnarray*}
\langle [ C,A ],B\rangle = \frac{1}{2} k_\beta \langle
H,[B_2,[C_1,A_1]]\rangle + \frac{1}{4} k_\alpha \langle H, [B_1 ,[A_2,C_1]]
\rangle \\
- \frac{1}{4} k_\gamma \langle H, [B_1 ,[C_2,A_1]] \rangle.
\end{eqnarray*}
Then, putting everything together: 
\begin{align*}
\Nij(A,B,C) & = \frac{1}{6} ( k_\gamma \langle
H,[C_2,[A_1,B_1]]\rangle + k_\alpha \langle H,[A_2,[B_1,C_1]]\rangle \\ 
& + k_\beta \langle H,[B_2,[C_1,A_1]]\rangle ) +  \frac{1}{12} ( k_\beta \langle H, [C_1 ,[B_2,A_1]] \rangle  \\
& + k_\gamma \langle H, [A_1 ,[C_2,B_1]] \rangle + k_\alpha \langle H, [B_1
,[A_2,C_1]] \rangle ) \\
&  -  \frac{1}{12} ( k_\alpha \langle H, [C_1 ,[A_2,B_1]] \rangle +
k_\beta \langle H, [A_1 ,[B_2,C_1]] \rangle \\ 
& + k_\gamma \langle H, [B_1
,[C_2,A_1]] \rangle ).
\end{align*}
Reorganizing the terms and using the Jacobi identity, we have that 
\begin{align}  \label{nijenhuis}
\Nij(A,B,C)& =  \frac{1}{12}( k_\gamma \langle H, [C_2 ,[A_1,B_1]]
\rangle + k_\alpha \langle H, [A_2 ,[B_1,C_1]] \rangle\nonumber\\
& + k_\beta \langle H,
[B_2 ,[C_1,A_1]] \rangle ).
\end{align}

\begin{remark}
Note that if $A_1 = A_2$, $B_1=B_2$ and $C_1=C_2$, by the Jacobi identity we
have $\Nij (A,B,C)=0$ whenever $k_\alpha = k_\beta = k_\gamma$.
\end{remark}

Moreover, by the expression (\ref{nijenhuis}) we have some immediate results.

\begin{cor}
\label{corol1} Let $A,B,C\in \mathfrak{u}$. Then 
\begin{equation*}
\Nij (A,B,C)=0.
\end{equation*}
\end{cor}

\begin{cor}
\label{corol2} Let $A,B \in \mathfrak{u}^\ast$. Then 
\begin{equation*}
\Nij (A,B,C)=0.
\end{equation*}
for all $C\in \mathfrak{u}\oplus \mathfrak{u}^\ast$.
\end{cor}

\subsection{Integrability}

We have already described all invariant generalized almost complex
structures in dimension $4$. In this way, for each $\alpha \in \Pi $,
consider $\mathcal{B}=\{A_{\alpha },S_{\alpha },-S_{\alpha }^{\ast
},A_{\alpha }^{\ast }\}$ basis of $\mathfrak{u}_{\alpha }\oplus \mathfrak{u}%
_{\alpha }^{\ast }$. Note that, with this basis the matrix of the bilinear
form $\langle \cdot ,\cdot \rangle $ is $Q=\left( 
\begin{array}{cc}
0 & 1 \\ 
1 & 0%
\end{array}%
\right) $. Then the only invariant generalized almost complex structures on $%
\mathfrak{u}_{\alpha }\oplus \mathfrak{u}_{\alpha }^{\ast }$ are 
\begin{equation*}
J=\pm J_{0}=\pm \left( 
\begin{array}{cccc}
0 & -1 & 0 & 0 \\ 
1 & 0 & 0 & 0 \\ 
0 & 0 & 0 & -1 \\ 
0 & 0 & 1 & 0%
\end{array}%
\right) \ \mathnormal{\ }\mathrm{or}\ J=\left( 
\begin{array}{cccc}
a_{\alpha } & 0 & 0 & -x_{\alpha } \\ 
0 & a_{\alpha } & x_{\alpha } & 0 \\ 
0 & -y_{\alpha } & -a_{\alpha } & 0 \\ 
y_{\alpha } & 0 & 0 & -a_{\alpha }%
\end{array}%
\right)
\end{equation*}%
where $x_{\alpha }y_{\alpha }\not=0$ and $a_{\alpha }^{2}-x_{\alpha
}y_{\alpha }=-1$, with $a_{\alpha },x_{\alpha },y_{\alpha }\in \mathbb{R}$.

Thus, we fix the following notation:

\begin{itemize}
\item[a)] $J = \pm J_0 = \pm \left( 
\begin{array}{cccc}
0 & -1 & 0 & 0 \\ 
1 & 0 & 0 & 0 \\ 
0 & 0 & 0 & -1 \\ 
0 & 0 & 1 & 0%
\end{array}
\right).$ In this case, we will say that $J$ is of complex type.

\item[b)] $J = \left( 
\begin{array}{cccc}
a_\alpha & 0 & 0 & -x_\alpha \\ 
0 & a_\alpha & x_\alpha & 0 \\ 
0 & -y_\alpha & -a_\alpha & 0 \\ 
y_\alpha & 0 & 0 & -a_\alpha%
\end{array}
\right),$ where $x_\alpha y_\alpha \not=0$ and $a_\alpha ^2-x_\alpha
y_\alpha = -1$, with $a_\alpha ,x_\alpha ,y_\alpha \in \mathbb{R}$. In this
case, we will say that $J$ is of non-complex type.
\end{itemize}

\begin{remark}
The nomenclature used above is due to the fact that if $J$ is of complex
type, then $J$ is a generalized almost complex structure coming from an
almost complex structure, that is, 
\begin{equation*}
J = \left( 
\begin{array}{cc}
-\widehat{J}_0 & 0 \\ 
0 & \widehat{J}_0 ^\ast%
\end{array}%
\right)
\end{equation*}
where $\widehat{J}_0$ is an almost complex structure. To know, the almost
complex structure is $\widehat{J}_0= \pm \left( 
\begin{array}{cc}
0 & 1 \\ 
-1 & 0%
\end{array}
\right)$.
\end{remark}

In this moment we are interested to study when they are integrable, that is,
when a generalized almost complex structure is, in fact, a generalized
complex structure. To do this, given $\mathcal{J}$ a generalized almost
complex structure, we are going to analyze the Nijenhuis operator restricted
to the $i$-eigenspace of $\mathcal{J}$.

With some simple calculations we can see the following facts.

\begin{itemize}
\item[a)] If $J$ is of complex type its $i$-eigenspace is 
\begin{equation*}
L=\mathrm{span}_{\mathbb{C}}\{A_{\alpha }-iS_{\alpha },A_{\alpha }^{\ast
}-iS_{\alpha }^{\ast }\}
\end{equation*}%
or 
\begin{equation*}
L=\mathrm{span}_{\mathbb{C}}\{A_{\alpha }+iS_{\alpha },A_{\alpha }^{\ast
}+iS_{\alpha }^{\ast }\}
\end{equation*}%
depending on whether we have $J=J_{0}$ or $J=-J_{0}$, respectively.

\item[b)] If $J$ is of non-complex type its $i$-eigenspace is 
\begin{equation*}
L=\mathrm{span}_{\mathbb{C}}\{x_{\alpha }A_{\alpha }+(a_{\alpha
}-i)A_{\alpha }^{\ast },x_{\alpha }S_{\alpha }+(a_{\alpha }-i)S_{\alpha
}^{\ast }\}.
\end{equation*}
\end{itemize}

As we will do calculations involving elements $A_\alpha$ and $S_\alpha$,
with $\alpha$ being a root, worth remembering the following:

\begin{lem}
The Lie bracket between the basic elements of $\mathfrak{u}$ are given by: 
\begin{multicols}{2}\noindent
$[iH_\alpha, A_\beta] = \beta(H_\alpha)S_\beta$ \\
$[iH_\alpha, S_\beta] = - \beta(H_\alpha)A_\beta$\\
$[A_\alpha,S_\alpha]=2iH_\alpha$\\
$[A_\alpha ,A_\beta]=m_{\alpha ,\beta}A_{\alpha +\beta}+m_{-\alpha ,\beta}A_{\alpha -\beta}$\\
$[S_\alpha ,S_\beta]=-m_{\alpha ,\beta}A_{\alpha +\beta}-m_{\alpha ,-\beta}A_{\alpha -\beta}$\\
$[A_\alpha ,S_\beta]=m_{\alpha ,\beta}S_{\alpha +\beta}+m_{\alpha ,-\beta}S_{\alpha -\beta}$.
\end{multicols}
\end{lem}

Now, in search of integrable structures we will consider the restriction of
the Nijenhuis operator to the subspace spanned by $\{A_{\alpha },S_{\alpha
},A_{\alpha }^{\ast },S_{\alpha }^{\ast }\}$.

\begin{pps}
Let $\alpha $ be a root. The Nijenhuis operator restricted to $\mathfrak{u}%
_{\alpha }\oplus \mathfrak{u}_{\alpha }^{\ast }$ is identically zero, where $%
\mathfrak{u}_{\alpha }=\mathrm{span}_{\mathbb{C}}\{A_{\alpha },S_{\alpha }\}$
and $\mathfrak{u}_{\alpha }^{\ast }=\mathrm{span}_{\mathbb{C}}\{A_{\alpha
}^{\ast },S_{\alpha }^{\ast }\}$.
\end{pps}

\begin{proof}
By corollary \ref{corol2}, we just need to analyse two cases:
\begin{enumerate}
\item $\Nij(A_\alpha,S_\alpha,A^\ast _\alpha)$
\begin{eqnarray*}
\Nij(A_\alpha,S_\alpha,A^\ast _\alpha) & = & \frac{1}{12} k_\alpha \left( \langle H,[A_\alpha,[A_\alpha,S_\alpha]]\rangle \right)\\
& = & \frac{1}{12} k_\alpha \langle H,[A_\alpha, 2iH_\alpha]\rangle \\
& = & -\frac{i}{6} k_\alpha \langle H,\alpha(H_\alpha)S_\alpha\rangle\\
& = & 0, \textnormal{ because $H$ is orthogonal to $S_\alpha$.}
\end{eqnarray*}

\item $\Nij(A_\alpha,S_\alpha,S^\ast _\alpha)$
\begin{eqnarray*}
\Nij(A_\alpha,S_\alpha,S^\ast _\alpha) & = & \frac{1}{12} k_\alpha \left( \langle H,[S_\alpha,[A_\alpha,S_\alpha]]\rangle \right)\\
& = & \frac{1}{12} k_\alpha \langle H,[A_\alpha, 2iH_\alpha]\rangle \\
& = & -\frac{i}{6} k_\alpha \langle H,\alpha(H_\alpha)A_\alpha\rangle\\
& = & 0, \textnormal{ because $H$ is orthogonal to $A_\alpha$.}
\end{eqnarray*}
\end{enumerate}
Proving that $\Nij|_{\mathfrak{u}_\alpha \oplus \mathfrak{u}_\alpha ^\ast}=0$.
\end{proof}

This way, we are going to restrict the Nijenhuis operator to the subspace
spanned by $A_\alpha$, $S_\alpha$, $A^\ast _\alpha$, $S^\ast _\alpha$, $%
A_\beta$, $S_\beta$, $A^\ast _\beta$ e $S^\ast _\beta$, where $\alpha,\beta
\in \Pi$. Due to the propositions \ref{corol1} and \ref{corol2}, we just have 10
cases to check. Doing the calculations exactly equal we have done in the
last proposition, we have that: 
\begin{equation*}
\Nij(A_\alpha,S_\alpha,A^\ast _\beta) = \Nij(A_\alpha,S_\alpha,S^\ast
_\beta) = \Nij(A_\alpha,A^\ast _\alpha,A _\beta) = \Nij(A_\alpha,A^\ast
_\alpha,S _\beta)
\end{equation*}
\begin{equation*}
= \Nij(A_\alpha,S^\ast _\alpha,A _\beta) = \Nij(A_\alpha,S^\ast _\alpha,S
_\beta) = \Nij(S_\alpha,A^\ast _\alpha,A _\beta) = \Nij(S_\alpha,A^\ast
_\alpha,S _\beta)
\end{equation*}
\begin{equation*}
= \Nij(S_\alpha,S^\ast _\alpha,A _\beta) = \Nij(S_\alpha,S^\ast _\alpha,S
_\beta)=0.
\end{equation*}
From these calculations we can conclude that the Nijenhuis operator
restricted to the subspace spanned by two roots is zero.

\begin{pps}
Let $\alpha$ and $\beta$ be roots. The Nijenhuis operator restricted to $%
\mathfrak{u}_\alpha \oplus \mathfrak{u}_\beta \oplus \mathfrak{u}_\alpha
^\ast \oplus \mathfrak{u}_\beta ^\ast$ is zero.
\end{pps}

It remains to see the case in which we have three roots $\alpha $, $\beta $
and $\gamma $. We can split this in two cases: $\alpha +\beta +\gamma =0$
and $\alpha +\beta +\gamma \not=0$. But, with similar calculations, we have
that if $\alpha +\beta +\gamma \not=0$ then $\Nij$ is zero, because the only
case in which $\langle H,X\rangle \not=0$ is that $X\in i\mathfrak{h}_{%
\mathbb{R}}$ and this happens just when we have the bracket between $%
A_{\lambda }$ and $S_{\lambda }$, for some root $\lambda $.

Therefore we need to do the calculations with three roots $\alpha $, $\beta $
and $\gamma $ satisfying $\alpha +\beta +\gamma =0$. Let $\alpha ,\beta $ be
roots such that $\alpha +\beta $ is a root as well. After some calculations,
we have 
\begin{eqnarray*}
\Nij(A_{\alpha },S_{\beta },A_{\alpha +\beta }^{\ast }) &=&-\Nij(A_{\alpha
},A_{\beta },S_{\alpha +\beta }^{\ast })=\Nij(S_{\alpha },S_{\beta
},S_{\alpha +\beta }^{\ast }) \\
&=&\Nij(S_{\alpha },A_{\beta },A_{\alpha +\beta }^{\ast })=\frac{i}{6}%
m_{\alpha ,\beta }
\end{eqnarray*}

\begin{eqnarray*}
\Nij (A_\alpha, S^\ast _\beta, A_{\alpha+\beta}) & = & -\Nij (A_\alpha,
A^\ast _\beta, S_{\alpha+\beta}) = \Nij (S_\alpha, S^\ast _\beta, S
_{\alpha+\beta}) \\
& = & \Nij (S_\alpha, A^\ast _\beta, A_{\alpha+\beta}) = -\frac{i}{6}%
m_{-(\alpha+\beta),\alpha}
\end{eqnarray*}

\begin{eqnarray*}
\Nij(A_{\alpha }^{\ast },S_{\beta },A_{\alpha +\beta }) &=&-\Nij(A_{\alpha
}^{\ast },A_{\beta },S_{\alpha +\beta })=\Nij(S_{\alpha }^{\ast },S_{\beta
},S_{\alpha +\beta }) \\
&=&\Nij(S_{\alpha }^{\ast },A_{\beta },A_{\alpha +\beta })=-\frac{i}{6}%
m_{\beta ,-(\alpha +\beta )}
\end{eqnarray*}%
and the other cases are all zero. For these calculations we recall that $%
m_{\alpha ,\beta }=m_{\beta ,\gamma }=m_{\gamma ,\alpha }$ when $\alpha
+\beta +\gamma =0$.

Now we will consider an invariant generalized almost complex structure $\mathcal{J}$,
and we know that for each root $\alpha $ the restriction of $\mathcal{J}$ to 
$\mathfrak{u}_{\alpha }\oplus \mathfrak{u}_{\alpha }^{\ast }$ is of complex
or non-complex type, this implies that the $i$-eigenspace of $\mathcal{J}$
restrict to this subspace is $L_{\alpha }=\mathrm{span}_{\mathbb{C}%
}\{A_{\alpha }\pm iS_{\alpha },A_{\alpha }^{\ast }\pm iS_{\alpha }^{\ast }\}$
or $L_{\alpha }=\mathrm{span}_{\mathbb{C}}\{x_{\alpha }A_{\alpha
}+(a_{\alpha }-i)A_{\alpha }^{\ast },x_{\alpha }S_{\alpha }+(a_{\alpha
}-i)S_{\alpha }^{\ast }\}$. So we need to analyze the Nijenhuis operator
restricted to the $i$-eigenspace of $\mathcal{J}$ for each triple of roots $%
\alpha $, $\beta $ and $\alpha +\beta $.

First of all let us fix some notation. Let $\mathcal{J}$ be an invariant generalized
almost complex structure and let $\alpha $ and $\beta $ be roots such that $%
\alpha +\beta $ is a root too. We will denote by $J_{\alpha }$, $J_{\beta }$
and $J_{\alpha +\beta }$ the restriction of $\mathcal{J}$ to the roots $%
\alpha $, $\beta $ and $\alpha + \beta$ respectively. And we will denote by $%
L_{\alpha }$, $L_{\beta }$ and $L_{\alpha +\beta }$ the $i$-eingenspace of $%
J_{\alpha }$, $J_{\beta }$ and $J_{\alpha +\beta }$ respectively.

\begin{pps}
Given $\mathcal{J}$ a generalized almost complex strucutre. Suppose that $%
J_{\alpha }$ and $J_{\beta }$ are of complex type both with the same sign,
that is, both equal to $J_{0}$ or $-J_{0}$, and suppose $J_{\alpha +\beta }$
of non-complex type. Then the Nijenhuis operator restricted to $L=L_{\alpha
}\cup L_{\beta }\cup L_{\alpha +\beta }$ is nonzero.
\end{pps}

\begin{proof}
It follows from a direct calculation:
\begin{itemize}
\item[a)] If $J_\alpha$ and $J_\beta$ are equal $J_0$:
\begin{eqnarray*}
\Nij (A_\alpha - i S_\alpha,A^\ast _\beta -iS^\ast _\beta, x_{\alpha+\beta} A_{\alpha+\beta}+(a_{\alpha+\beta}-i)A^\ast _{\alpha+\beta}) \\ =-\frac{1}{3}x_{\alpha+\beta}m_{-(\alpha+\beta),\alpha}.
\end{eqnarray*}

\item[b)] If $J_\alpha$ and $J_\beta$ are equal $-J_0$:
\begin{eqnarray*}
\Nij (A_\alpha + i S_\alpha,A^\ast _\beta +iS^\ast _\beta, x_{\alpha+\beta} A_{\alpha+\beta}+(a_{\alpha+\beta}-i)A^\ast _{\alpha+\beta}) \\ =\frac{1}{3}x_{\alpha+\beta}m_{-(\alpha+\beta),\alpha}.
\end{eqnarray*}
\end{itemize}
\end{proof}

We can prove a similar result as follows.

\begin{pps}
Given $\mathcal{J}$ a generalized almost complex strucutre. Suppose that two
of $J_\alpha$, $J_\beta$ and $J_{\alpha+\beta}$ are of non-complex type and
the other one of complex type, then the Nijenhuis operator restricted to $L
= L_\alpha \cup L_\beta \cup L_{\alpha+\beta}$ is nonzero.
\end{pps}

\begin{proof}
Again, we will do the calculations case by case:
\begin{itemize}
\item[a)] If $J_\alpha$ and $J_\beta$ are of non-complex type, then $J_{\alpha+\beta}$ is of complex type. For $J_{\alpha+\beta}$ we can have $J_0$ and $-J_0$. Then, we have
\begin{eqnarray*}
\Nij (x_\alpha A_\alpha+(a_\alpha -i)A^\ast _\alpha, x_\beta A_\beta+(a_\beta -i)A^\ast _\beta, A^\ast _{\alpha+\beta} - iS^\ast _{\alpha +\beta}) \\ =-\frac{1}{6}x_\alpha x_\beta m_{\alpha,\beta}
\end{eqnarray*}
when $J_{\alpha+\beta} = J_0$ and
\begin{eqnarray*}
\Nij (x_\alpha A_\alpha+(a_\alpha -i)A^\ast _\alpha, x_\beta A_\beta+(a_\beta -i)A^\ast _\beta, A^\ast _{\alpha+\beta}+iS^\ast _{\alpha +\beta}) \\ =\frac{1}{6}x_\alpha x_\beta m_{\alpha,\beta}
\end{eqnarray*}
when $J_{\alpha+\beta}=-J_0$.

\item[b)] If $J_\alpha$ and $J_{\alpha+\beta}$ is of non-complex type and $J_\beta$ is of complex type. Over again, we can have $J_\beta$ equals to $J_0$ or $-J_0$. This way, we have
\begin{eqnarray*}
\Nij (x_\alpha A_\alpha+(a_\alpha -i)A^\ast _\alpha,A^\ast _\beta - iS^\ast _\beta, x_{\alpha+\beta}A_{\alpha+\beta}+(a_{\alpha+\beta}-i)A^\ast _{\alpha +\beta}) \\ =-\frac{1}{6}x_\alpha x_{\alpha+\beta} m_{-(\alpha+\beta),\alpha}
\end{eqnarray*}
when $J_\beta = J_0$ and
\begin{eqnarray*}
\Nij (x_\alpha A_\alpha+(a_\alpha -i)A^\ast _\alpha,A^\ast _\beta + iS^\ast _\beta, x_{\alpha+\beta}A_{\alpha+\beta}+(a_{\alpha+\beta}-i)A^\ast _{\alpha +\beta}) \\ =\frac{1}{6}x_\alpha x_{\alpha+\beta} m_{-(\alpha+\beta),\alpha}
\end{eqnarray*}
when $J_\beta =-J_0$.
\end{itemize}
\end{proof}

\begin{remark}
Observe that the case where $J_\beta$ and $J_{\alpha+\beta}$ is of
non-complex type and $J_\alpha$ is of complex one, is exactly the same case
of the item (b), because we can see $\alpha+\beta = \beta+\alpha$.
\end{remark}

\begin{pps}
Given $\mathcal{J}$ a generalized almost complex strucutre. Suppose that $%
J_\alpha$ is of non-complex type and $J_\beta$, $J_{\alpha+\beta}$ are of
complex type with different signs, that is, if $J_\beta=J_0$ then $%
J_{\alpha+\beta} = -J_0$ or vice versa. Then the Nijenhuis operator
restricted to $L = L_\alpha \cup L_\beta \cup L_{\alpha+\beta}$ is nonzero.
\end{pps}

\begin{proof}
We have two cases to analyze. First one, suppose $J_\beta=J_0$ and $J_{\alpha+\beta}=-J_0$. In this case, we have
\[
\Nij (x_\alpha A_\alpha +(a_\alpha -i)A^\ast _\alpha,A _\beta - iS_\beta, A^\ast _{\alpha+\beta}+ iS^\ast _{\alpha +\beta}) =\frac{1}{3}x_\alpha m_{\alpha,\beta}.
\]
For the second case, suppose $J_\beta = -J_0$ and $J_{\alpha+\beta}=J_0$. Then
\[
\Nij (x_\alpha A_\alpha +(a_\alpha -i)A^\ast _\alpha,A _\beta + iS_\beta, A^\ast _{\alpha+\beta}- iS^\ast _{\alpha +\beta}) =-\frac{1}{3}x_\alpha m_{\alpha,\beta}.
\] 
\end{proof}

Note that, this proposition shows that when $J_\alpha$ and $J_{\alpha+\beta}$
are of complex type with different signs and $J_\beta$ is of non-complex
type, then the Nijenhuis operator is nonzero, for the same argument used in
the last remark. \newline

Observe that, if $J_\alpha$ is of complex type with $J_\alpha = J_0$, then
we have $J_\alpha (A_\alpha) = -S_\alpha$ and $J_\alpha (S_\alpha) =
A_\alpha $. It follows that $J_\alpha (X_\alpha) = -iX_\alpha$ and $J_\alpha
= iX_{-\alpha}$, which means that $\varepsilon _\alpha = -1$. Analogously,
if $J_\alpha = -J_0$ then $\varepsilon _\alpha = 1$.

In Gualtieri \cite{Gualtieri} it is proved that given a generalized almost
complex structure $\mathcal{J}$ which comes from an almost complex structure 
$J$, then $\mathcal{J}$ is integrable if and only if $J$ is integrable.


\begin{teo}
With the above hypotheses, in the cases which
\begin{center}
\begin{tabular}{|c|c|c|}
\hline
$J_\alpha$ & $J_\beta$ & $J_{\alpha+\beta}$ \\ \hline
$J_0$ & $J_0$ & $J_0$ \\ \hline
$J_0$ & $-J_0$ & $J_0$ \\ \hline
$-J_0$ & $J_0$ & $-J_0$ \\ \hline
$-J_0$ & $-J_0$ & $-J_0$ \\ \hline
\end{tabular}
\end{center}
the Nijenhuis operator restricted to $L= L_\alpha \cup L_\beta \cup
L_{\alpha+\beta}$ is zero.
\end{teo}

\begin{remark}
A consequence of this is that in case $J_{\alpha }=J_{\beta }=\pm J_{0}$ and 
$J_{\alpha +\beta }=\mp J_{0}$ then the Nijenhuis operator restricted to $L$
is nonzero.
\end{remark}

\begin{teo}
Given $\mathcal{J}$ a generalized almost complex structure. Suppose that $%
J_\alpha$ is of non-complex type and $J_\beta$, $J_{\alpha+\beta}$ are of
complex type both with the same sign. Then the Nijenhuis operator restricted
to $L = L_\alpha \cup L_\beta \cup L_{\alpha+\beta}$ is zero. The same is
true if $J_\alpha,J_\beta$ are of complex type with different signs and $%
J_{\alpha+\beta}$ is of non-complex type.
\end{teo}

\begin{proof}
The proof is a direct calculation of the Nijenhuis operator restricted to $L = L_\alpha \cup L_\beta \cup L_{\alpha+\beta}$, using all that we have until here. More specifically, we use the multilinearity of the Nijenhuis operator and its expression for the basic elements  previously calculated. 
\end{proof}

Among all possible cases, remains the case in which $J_\alpha$, $J_\beta$
and $J_{\alpha+\beta}$ are of non-complex type. Calculating the Nijenhuis
operator restricted to $L = L_\alpha \cup L_\beta \cup L_{\alpha+\beta}$, we
obtain that $\Nij|_L = 0$ if and only if 
\begin{equation}  \label{nij333}
x_\alpha x_\beta(a_{\alpha+\beta}-i)-x_\alpha x_{\alpha+\beta}(a_\beta -
i)-x_\beta x_{\alpha+\beta} (a_\alpha - i)=0.
\end{equation}
Note that, the expression (\ref{nij333}) is equivalent to 
\begin{equation*}
\left\lbrace 
\begin{array}{l}
a_{\alpha+\beta} x_\alpha x_\beta - a_\beta x_\alpha x_{\alpha+\beta} -
a_\alpha x_\beta x_{\alpha+\beta} = 0 \\ 
x_\alpha x_\beta -x_\alpha x_{\alpha+\beta} - x_\beta x_{\alpha+\beta} = 0%
\end{array}
\right.
\end{equation*}
which allows us to state the following:

\begin{teo}
\label{teo333} Given $\mathcal{J}$ a generalized almost complex strucutre.
Suppose that $J_\alpha$, $J_\beta$ and $J_{\alpha+\beta}$ are of non-complex
type. Then the Nijenhuis operator restricted to $L = L_\alpha \cup L_\beta
\cup L_{\alpha+\beta}$ is zero if and only if 
\begin{equation*}
\left\lbrace 
\begin{array}{l}
a_{\alpha+\beta} x_\alpha x_\beta - a_\beta x_\alpha x_{\alpha+\beta} -
a_\alpha x_\beta x_{\alpha+\beta} = 0 \\ 
x_\alpha x_\beta -x_\alpha x_{\alpha+\beta} - x_\beta x_{\alpha+\beta} = 0%
\end{array}
\right.
\end{equation*}
where, for each $\gamma \in \{ \alpha,\beta, \alpha+\beta\}$, we have $%
a_\gamma ^2= x_\gamma y_\gamma -1$.
\end{teo}

Before continuing, for any root $\alpha$ we are considering the basis $%
\mathcal{B}= \{A_\alpha,S_\alpha, -S_\alpha ^\ast, A_\alpha ^\ast\}$ for $%
\mathfrak{u}_\alpha \oplus \mathfrak{u}_\alpha ^\ast$. Observe that $%
\mathcal{B}^{\prime }=\{A_{-\alpha}, S_{-\alpha}, -S_{-\alpha}
^\ast,A_{-\alpha} ^\ast \}$ is also a basis for $\mathfrak{u}_\alpha \oplus 
\mathfrak{u}_\alpha ^\ast$. The base change matrix from $\mathcal{B}$ to $%
\mathcal{B}^{\prime }$ is 
\begin{equation*}
M = \left( 
\begin{array}{cccc}
-1 & 0 & 0 & 0 \\ 
0 & 1 & 0 & 0 \\ 
0 & 0 & -1 & 0 \\ 
0 & 0 & 0 & 1%
\end{array}
\right) = M^{-1}.
\end{equation*}
Thereby, if $J_\alpha$ is of complex type, with $J_\alpha = J_0$ in the
basis $\mathcal{B}$, changing for the basis $\mathcal{B}^{\prime }$ we
obtain $-J_0$. And, if $J_\alpha$ is of non-complex type in the basis $%
\mathcal{B}$, with $x_\alpha$, $y_\alpha$ and $a_\alpha$ satisfying $%
a_\alpha ^2=x_\alpha y_\alpha -1$, then changing for the basis $\mathcal{B}%
^{\prime }$ we obtain another matrix of non-complex type 
\begin{equation*}
\left( 
\begin{array}{cccc}
a_{-\alpha} & 0 & 0 & -x_{-\alpha} \\ 
0 & a_{-\alpha} & x_{-\alpha} & 0 \\ 
0 & -y_{-\alpha} & -a_{-\alpha} & 0 \\ 
y_{-\alpha} & 0 & 0 & -a_{\alpha}%
\end{array}
\right)
\end{equation*}
where $x_{-\alpha}=-x_\alpha$, $y_{-\alpha} = -y_\alpha$ and $a_{-\alpha} =
a_\alpha$.

Summarizing, from what was done above, the change of basis from $\mathcal{B}$
to $\mathcal{B}^{\prime }$ doesn't change the type of the structure. But,
change some signs, more specifically, if is of complex type then change from 
$J_0$ to $-J_0$ and if is of non-complex type change the signs of $x_\alpha$
and $y_\alpha$.

\begin{pps}
\label{rootsystem} Let $\mathcal{J}$ be a generalized complex structure.
Then the set 
\begin{eqnarray*}
P = \{ \alpha \in \Pi \ | \ J_\alpha \mathit{\ is \ of \ complex \ type \ with \ }
J_\alpha = J_0\} \cup \\
\{ \alpha \in \Pi \ | \ J_\alpha \mathit{\ is \ of \ non-complex \ type \ with \ }
x_\alpha >0 \}
\end{eqnarray*}
is a choice of positive roots with respect to some lexicographic order in $%
\mathfrak{h}_\mathbb{R} ^\ast$.
\end{pps}

\begin{proof}
Since $\mathcal{J}$ is integrable, let's prove that $P$ is a closed set. For this, we need to prove that given $\alpha,\beta \in P$ such that $\alpha +\beta$ is a root, then $\alpha +\beta \in P$. Thus, let $\alpha,\beta \in P$ such that $\alpha+\beta$ is a root. We need to analyze the possibilities:
\begin{itemize}
\item If $J_\alpha$ and $J_\beta$ are of complex type with $J_\alpha = J_\beta = J_0$, then the only possibility is that $J_{\alpha+\beta}$ is of complex type with $J_{\alpha+\beta}=J_0$. Therefore $\alpha+\beta \in P$;

\item If $J_\alpha$ is of complex type with $J_\alpha = J_0$ and $J_\beta$ is of non-complex type with $x_\beta >0$, then the only possibility is that $J_\alpha$ is of complex type with $J_{\alpha+\beta} = J_0$. Therefore $\alpha+\beta \in P$;

\item If $J_\alpha$ is of non-complex type with $x_\alpha >0$ and $J_\beta$ is of complex type with $J_\beta = J_0$, is analogous to the previous case;

\item If $J_\alpha$ and $J_\beta$ are of non-complex type with $x_\alpha,x_\beta >0$, then the only possibility is that $J_{\alpha+\beta}$ is of non-complex type too. Now, as $\mathcal{J}$ is integrable, we have
\[
x_\alpha x_\beta -x_\alpha x_{\alpha+\beta} - x_\beta x_{\alpha+\beta} = 0
\] 
from where we get 
\[
x_{\alpha+\beta} = \frac{x_\alpha x_\beta}{x_\alpha + x_\beta}>0
\]
because $x_\alpha,x_\beta >0$. Therefore $\alpha+\beta \in P$.
\end{itemize}
Proving that $P$ is closed. Furthermore, note that $\Pi = P \cup (-P)$. Then, we know that these properties imply that $P$ is a choice of positive roots.
\end{proof}

In particular, we have:

\begin{cor}
Let $\mathcal{J}$ be a generalized complex structure. If $J_\alpha$ is of
non-complex type for all $\alpha$, then $P =\{ \alpha \in \Pi \ | \ x_\alpha
>0 \}$ is a choice of positive roots with respect to some lexicographic
order in $\mathfrak{h}_\mathbb{R} ^\ast$.
\end{cor}

Given $\mathcal{J}$ a generalized complex structure, note that the set of
roots $\alpha$ such that $J_\alpha$ is of non-complex type is a closed set.
Indeed, the integrability ensures that for all $\alpha, \beta$ roots, such
that $\alpha +\beta$ is a root, with $J_\alpha$ and $J_\beta$ of non-complex
type, then $J_{\alpha+\beta}$ must be of non-complex type.

\begin{teo}
\label{existetheta} Let $\mathcal{J}$ be a generalized complex structure.
Then there is a subset $\Theta \subset \Sigma$, where $\Sigma$ is a simple
root system, such that $J_\alpha$ is of non-complex type for each $\alpha
\in \langle \Theta \rangle ^+$ and of complex type for $\alpha \in \Pi^+
\backslash \langle \Theta \rangle ^+$.
\end{teo}

\begin{proof}
It follows from the previous comment and proposition \ref{rootsystem}.  
\end{proof}

The following statement is a converse to the above theorem.

\begin{teo}
\label{dadotheta} Let $\Sigma$ be a simple root system and consider $\Theta
\subset \Sigma$ a subset. Then there is a generalized complex structure, $%
\mathcal{J}$, such that $J_\alpha$ is of non-complex type for $\alpha \in
\langle \Theta \rangle ^+$ and complex type for $\alpha \in \Pi ^+
\backslash \langle \Theta \rangle ^+$.
\end{teo}

\begin{proof}
Since $\langle \Theta \rangle$ is a closed set, then $\langle \Theta \rangle ^+$ is closed too. Thus, given $\alpha,\beta \in \langle \Theta \rangle ^+$ with $\alpha + \beta \in \Pi$, we have $\alpha + \beta \in \langle \Theta \rangle ^+$. Therefore, we can put $J_\alpha$, $J_\beta$ and $J_{\alpha+\beta}$ of non-complex type satisfying 
\[
\left\lbrace \begin{array}{l}
a_{\alpha+\beta} x_\alpha x_\beta - a_\beta x_\alpha x_{\alpha+\beta} - a_\alpha x_\beta x_{\alpha+\beta} = 0 \\
x_\alpha x_\beta - x_\alpha x_{\alpha+\beta} - x_\beta x_{\alpha+\beta} = 0.
\end{array}\right.
\]
This way the Nijenhuis operator restricted to $L = L_\alpha \cup L_\beta \cup L_{\alpha + \beta}$ is zero.

In the cases which $\alpha,\beta \in \Pi^+ \backslash \langle \Theta \rangle ^+$ with $\alpha + \beta \in \Pi$, we have $\alpha + \beta \in \Pi^+ \backslash \langle \Theta \rangle ^+$. Then, for these roots, we put $J_\alpha, J_\beta, J_{\alpha+\beta}$ of complex type with $J_\alpha = J_\beta = J_{\alpha + \beta} = J_0$, which defines a structure satisfying $\Nij|_L = 0$, where $L = L_\alpha \cup L_\beta \cup L_{\alpha + \beta}$

Now, in the case which $\alpha \in \Pi^+ \backslash \langle \Theta \rangle ^+$ and $\beta \in \langle \Theta \rangle ^+$, with $\alpha + \beta \in \Pi$, we must have $\alpha + \beta \in \Pi^+ \backslash \langle \Theta \rangle ^+$. In fact, suppose $\alpha$ and $\beta$ as before, such that $\alpha + \beta \in 
\langle \Theta \rangle ^+$. Since $\beta \in  \langle \Theta \rangle$, we have $-\beta \in  \langle \Theta \rangle$, and then $\alpha = (\alpha + \beta) + (-\beta) \in  \langle \Theta \rangle$ which contradicts the hypothesis. Therefore, in this case, we have $\alpha + \beta \in \Pi^+ \backslash \langle \Theta \rangle ^+$. This way, we put $J_\alpha$ and $J_{\alpha+ \beta}$ of complex type with $J_\alpha = J_{\alpha + \beta} = J_0$ and $J_{\beta}$ of non-complex type. Thus defining a structure whose Nijenhuis operator is zero when restricted to the $i$-eigeinspace.

Therefore, $\mathcal{J}$ defined in this way is an integrable generalized complex structure.  
\end{proof}

\subsection{Solutions for the non-complex type}

Let $\mathcal{J}$ be a generalized almost complex structure such that there
exist $\alpha $, $\beta $ and ${\alpha +\beta }$ roots that $J_{\alpha }$, $%
J_{\beta }$ and $J_{\alpha +\beta }$ are of non-complex type. We saw that
the Nijenhuis operator restricted to $L=L_{\alpha }\cup L_{\beta }\cup
L_{\alpha +\beta }$ is zero if and only if the equations 
\begin{equation}
\left\{ 
\begin{array}{l}
a_{\alpha +\beta }x_{\alpha }x_{\beta }-a_{\beta }x_{\alpha }x_{\alpha
+\beta }-a_{\alpha }x_{\beta }x_{\alpha +\beta }=0 \\ 
x_{\alpha }x_{\beta }-x_{\alpha }x_{\alpha +\beta }-x_{\beta }x_{\alpha
+\beta }=0%
\end{array}%
\right.  \label{sistema}
\end{equation}%
are satisfied.

We are interested to know when there are $x_\alpha$, $x_\beta$, $x_{\alpha +
\beta}$, $a_\alpha$, $a_\beta$ e $a_{\alpha + \beta}$ which satisfy the
system (\ref{sistema}), where $x_\alpha,x_\beta,x_{\alpha + \beta} \not= 0$.
Manipulating the system we get that the system is satisfied when 
\begin{equation*}
x_{\alpha + \beta}=\frac{x_\alpha x_\beta}{x_\alpha + x_\beta} \ \mathrm{
\ and } \ a_{\alpha + \beta} = \frac{a_\beta x_\alpha + a_\alpha x_\beta}{%
x_\alpha + x_\beta}.
\end{equation*}
Thus, given $a_\alpha$, $a_\beta$, $x_\alpha$, $x_\beta$ the only solution
for (\ref{sistema}) is $a_{\alpha + \beta}$ and $x_{\alpha + \beta}$ as
above.

The problem is when we have more than one such triple of roots. To explain
this, let's do an example:

\begin{ex}
Consider the Lie algebra $A_{3}$ with simple roots $\Sigma =\{\alpha ,\beta
,\gamma \}$ and positive roots $\alpha +\beta $, $\beta +\gamma $ and $%
\alpha +\beta +\gamma $. Let $\mathcal{J}$ be a generalized almost complex
structure such that $J_{\delta }$ is of non-complex type for all root $%
\delta $. We want that $\mathcal{J}$ be integrable, that is, for all triple
of roots $\alpha _{1}$, $\alpha _{2}$ and $\alpha _{1}+\alpha _{2}$, we need
that the system (\ref{sistema}) be satisfied. Our idea is to take values for 
$J_{\delta }$, where $\delta $ is a simple root, and find for which values
of $J_{\eta }$ the system is satisfied, where $\eta $ is a positive root of
height greater than one. In this example, we have that the roots of height
two can be written uniquely as sum of two simple roots, but the root $\alpha
+\beta +\gamma $ can be written of two ways as sum of two roots, $(\alpha
+\beta )+\gamma $ and $\alpha +(\beta +\gamma )$. Then we have to be
careful, because all possible system must be satisfied.

Fix $a_{\delta }$, $x_{\delta }$ for all root $\delta \in \Sigma $. We will
analyse all possible system: First, consider the triple $\alpha $, $\beta $
and $\alpha +\beta $. We have the system 
\begin{equation*}
\left\{ 
\begin{array}{l}
a_{\alpha +\beta }x_{\alpha }x_{\beta }-a_{\beta }x_{\alpha }x_{\alpha
+\beta }-a_{\alpha }x_{\beta }x_{\alpha +\beta }=0 \\ 
x_{\alpha }x_{\beta }-x_{\alpha }x_{\alpha +\beta }-x_{\beta }x_{\alpha
+\beta }=0,%
\end{array}%
\right.
\end{equation*}%
which has solution 
\begin{equation*}
x_{\alpha +\beta }=\frac{x_{\alpha }x_{\beta }}{x_{\alpha }+x_{\beta }}\ 
\mathrm{\ and}\ a_{\alpha +\beta }=\frac{a_{\beta }x_{\alpha }+a_{\alpha
}x_{\beta }}{x_{\alpha }+x_{\beta }}.
\end{equation*}%
Now, consider the triple $\beta $, $\gamma $ and $\beta +\gamma $. We have
the same system as above changing $\alpha $ by $\gamma $. Then, we have 
\begin{equation*}
x_{\beta +\gamma }=\frac{x_{\beta }x_{\gamma }}{x_{\beta }+x_{\gamma }}\ 
\mathrm{\ and}\ a_{\beta +\gamma }=\frac{a_{\gamma }x_{\beta }+a_{\beta
}x_{\gamma }}{x_{\beta }+x_{\gamma }}.
\end{equation*}%
For the triple $\alpha $, $\beta +\gamma $ and $\alpha +\beta +\gamma $, we
have a similar system which has solution 
\begin{equation*}
x_{\alpha +\beta +\gamma }=\frac{x_{\alpha }x_{\beta +\gamma }}{x_{\alpha
}+x_{\beta +\gamma }}\ \mathrm{\ and}\ a_{\alpha +\beta +\gamma }=\frac{%
a_{\beta +\gamma }x_{\alpha }+a_{\alpha }x_{\beta +\gamma }}{x_{\alpha
}+x_{\beta +\gamma }}.
\end{equation*}%
On the other hand, we have the triple $\alpha +\beta $, $\gamma $ and $%
\alpha +\beta +\gamma $, where we get another expression for $a_{\alpha
+\beta +\gamma }$ and $x_{\alpha +\beta +\gamma }$: 
\begin{equation*}
x_{\alpha +\beta +\gamma }=\frac{x_{\alpha +\beta }x_{\gamma }}{x_{\alpha
+\beta }+x_{\gamma }}\ \mathnormal{\ }\mathrm{and}\ a_{\alpha +\beta +\gamma
}=\frac{a_{\gamma }x_{\alpha +\beta }+a_{\alpha +\beta }x_{\gamma }}{%
x_{\alpha +\beta }+x_{\gamma }}.
\end{equation*}%
We have to verify that both expressions are equal. We already have
expressions for $x_{\beta +\gamma }$, $a_{\beta +\gamma }$, $x_{\alpha
+\beta }$ and $a_{\alpha +\beta }$. Replacing in the expressions above, we
have for both cases: 
\begin{equation*}
x_{\alpha +\beta +\gamma }=\frac{x_{\alpha }x_{\beta }x_{\gamma }}{x_{\alpha
}x_{\beta }+x_{\alpha }x_{\gamma }+x_{\beta }x_{\gamma }}
\end{equation*}%
and 
\begin{equation*}
a_{\alpha +\beta +\gamma }=\frac{a_{\gamma }x_{\alpha }x_{\beta }+a_{\beta
}x_{\alpha }x_{\gamma }+a_{\alpha }x_{\beta }x_{\gamma }}{x_{\alpha
}x_{\beta }+x_{\alpha }x_{\gamma }+x_{\beta }x_{\gamma }},
\end{equation*}%
proving that the expression is well defined. Therefore, setting the values
of $x_{\delta }$ and $a_{\delta }$, for the simple roots $\delta $, we can
obtain values for the other values of $x$ and $a$ associated to the positive
roots of height greater than one such that $\mathcal{J}$ is integrable.
\end{ex}

In the general case, let $\mathcal{J}$ be a generalized almost complex
structure. If $\mathcal{J}$ is integrable, by Theorem \ref{existetheta},
there is a $\Theta \subset \Sigma$ such that $J_\alpha$ is of non-complex
type for all $\alpha \in \langle \Theta \rangle$. Furthermore, by
Proposition \ref{rootsystem}, we can suppose that $x_\alpha >0$ for all $%
\alpha \in \langle \Theta \rangle ^+$. Again, we fix the values for $%
x_\alpha $ and $a_\alpha$ for all $\alpha \in \Theta$. We proceed solving
the system for the roots of height two, three, four and so on. Then, denote $%
\Theta = \{ \alpha_1, \cdots ,\alpha_l\}$ and fix $a_{\alpha_i},x_{\alpha_i}$
for $i=1,\cdots ,l$. Given a root $\alpha \in \langle \Theta \rangle$, then $%
\alpha$ can be written as $\alpha = n_1 \alpha_1 +\cdots + n_l \alpha_l$.
Proceeding as above, we obtain that 
\begin{equation}
x_{\alpha} = \frac{x_1 ^{n_1} \cdots x_l ^{n_l}}{\displaystyle\sum_{i=1} ^l
n_ix_{\alpha_1} ^{n_1}\cdots x_{\alpha_{i-1}} ^{n_{i-1}} x_{\alpha_i} ^{n_i
-1} x_{\alpha_{i+1}} ^{n_{i+1}} \cdots x_{\alpha_l} ^{n_l}}
\end{equation}
and 
\begin{equation}
a_{\alpha} = \frac{\displaystyle\sum_{i=1} ^l a_{\alpha_i}n_ix_{\alpha_1}
^{n_1}\cdots x_{\alpha_{i-1}} ^{n_{i-1}} x_{\alpha_i} ^{n_i -1}
x_{\alpha_{i+1}} ^{n_{i+1}} \cdots x_{\alpha_l} ^{n_l}}{\displaystyle%
\sum_{i=1} ^l n_ix_{\alpha_1} ^{n_1}\cdots x_{\alpha_{i-1}} ^{n_{i-1}}
x_{\alpha_i} ^{n_i -1} x_{\alpha_{i+1}} ^{n_{i+1}} \cdots x_{\alpha_l} ^{n_l}%
}
\end{equation}
is solution for the system for every $n_1, \cdots , n_l$ such that $n_1
\alpha_1 + \cdots + n_l \alpha_l$ is a root in $\langle \Theta \rangle$.


\section{Twisted Generalized Complex Structures on $\mathbb{F}$}

As were mentioned before, \v{S}evera and Weinstein noticed a twisted Courant
bracket $[\cdot ,\cdot ]_{\Omega }$ for each closed $3$-form $\Omega $ which
is defined by 
\begin{equation*}
\lbrack X+\xi ,Y+\eta ]_{\Omega }=[X+\xi ,Y+\eta ]+i_{Y}i_{X}\Omega .
\end{equation*}

Since $H^{3}(\mathbb{F})=\{0\}$ we have that every closed $3$-form is exact.
Because of this, let $\omega =\sum \omega _{\alpha ,\beta }X_{\alpha }^{\ast
}\wedge X_{\beta }^{\ast }$ be an invariant $2$-form. Relying on the
invariance of $\omega $ its exterior differential is easily computed from a
standard formula, then we have: 
\begin{equation*}
d\omega (X_{\alpha },X_{\beta },X_{\gamma })=-\omega ([X_{\alpha },X_{\beta
}],X_{\gamma })+\omega ([X_{\alpha },X_{\gamma }],X_{\beta })-\omega
([X_{\beta },X_{\gamma }],X_{\alpha }).
\end{equation*}

If we denote $\Omega = d\omega$, thus we have that $\Omega$ is an invariant $%
3$-form. Moreover, since $d^2 = 0$ we have that $\Omega$ is closed. From
this we have that $\Omega (X_\alpha,X_\beta,X_\gamma)$ is zero unless $%
\alpha+\beta+\gamma = 0$. In this case 
\begin{equation*}
\Omega (X_\alpha,X_\beta,X_\gamma) = m_{\alpha,\beta} \left(
\omega(X_\alpha,X_{-\alpha})+\omega (X_\beta,X_{-\beta})+\omega
(X_\gamma,X_{-\gamma}) \right).
\end{equation*}
With this, from a simple verification we have 
\begin{equation*}
\Omega (X_{-\alpha},X_{-\beta},X_{-\gamma}) = \Omega
(X_\alpha,X_\beta,X_\gamma)
\end{equation*}
since that $m_{\alpha,\beta} = -m_{-\alpha,-\beta}$ and $m_{\alpha,%
\beta}=m_{\beta,\gamma}=m_{\gamma,\alpha}$, where the last one is true when $%
\alpha +\beta+\gamma = 0$.

Now consider $A_\alpha = X_\alpha - X_{-\alpha}$ and $S_\alpha =
i(X_\alpha+X_{-\alpha})$. When we compute the value of $\Omega$ for this
elements, we obtain: 
\begin{eqnarray*}
\Omega (A_\alpha,A_\beta,S_{\alpha+\beta}) & = & -\Omega
(A_\alpha,S_\beta,A_{\alpha+\beta}) \\
& = & -\Omega (S_\alpha,A_\beta,A_{\alpha+\beta}) \\
& = & -\Omega (S_\alpha,S_\beta,S_{\alpha+\beta}) \\
& = & 2i\Omega (X_\alpha,X_\beta,X_{-(\alpha+\beta)})
\end{eqnarray*}
and the other cases are zero.

We are interested in verifying when $\Nij _\Omega |_L$ is zero, where $L$ is
the $i$-eigenspace of a generalized almost complex structure. By definition,
we have 
\begin{equation*}
\Nij_\Omega (X+\xi , Y+\eta,Z+\zeta ) = \Nij (X+\xi , Y+\eta,Z+\zeta )+
\Omega (X,Y,Z).
\end{equation*}
Then we can prove:

\begin{teo}
Let $\mathcal{J}$ be a generalized almost complex structure and $\Omega$ an
invariant closed $3$-form. Suppose there is a triple of roots $%
\alpha,\beta,\alpha+\beta$ such that $J_\alpha$, $J_\beta$ and $%
J_{\alpha+\beta}$ are not of non-complex type simultaneously. Then $\Nij|_L
= 0$ if and only if $\Nij_\Omega |_L = 0$, where $L= L_\alpha \cup L_\beta
\cup L_{\alpha+\beta}$.
\end{teo}

\begin{proof}
Suppose that $\Nij|_L=0$. To prove that $\Nij_\Omega |_L=0$ we need to prove that the restriction of $\Omega$ to the vector part of $L$ is zero. However, this is a consequence of 
\begin{eqnarray*}
\Omega (A_\alpha,A_\beta,S_{\alpha+\beta}) & = & -\Omega (A_\alpha,S_\beta,A_{\alpha+\beta})\\
& = & -\Omega (S_\alpha,A_\beta,A_{\alpha+\beta})\\
& = & -\Omega (S_\alpha,S_\beta,S_{\alpha+\beta})\\
& = & 2i\Omega (X_\alpha,X_\beta,X_{-(\alpha+\beta)}) 
\end{eqnarray*}
and the other cases are zero.

On the other hand, suppose $\Nij_\Omega |_L  = 0$. Since there is a triple of roots such that are not all of non-complex type, then there is $\delta \in \{\alpha,\beta,\alpha+\beta\}$ such that $J_\delta$ is of complex type and, in this case, we have $A_\delta ^\ast - iS_\delta ^\ast$ or $A_\delta ^\ast + iS_\delta ^\ast$ belongs to $L$. Therefore, we will have
\[
0 = \Nij_\Omega (X+\xi ,Y+\eta,A^\ast _\delta \pm iS^\ast _\delta  ) = \Nij (X+\xi ,Y+\eta,A^\ast _\delta \pm iS^\ast _\delta  ) 
\]
and, from the computations already done for the standard Nijenhuis operator, we have that if this occurs for all $X+\xi,Y+\eta \in L$, where $L=L_\alpha\cup L_\beta\cup L_{\alpha+\beta}$, then $\Nij |_L = 0$.
\end{proof}

Now, when there exists a triple of roots $\alpha,\beta,\alpha+\beta$ such
that $J_\alpha$, $J_\beta$ and $J_{\alpha+\beta}$ are all of non-complex
type, then we have:

\begin{teo}
Let $\mathcal{J}$ be a generalized almost complex structure and $\Omega$ an
invariant closed $3$-form. Suppose there is a triple of roots $%
\alpha,\beta,\alpha+\beta$ such that $J_\alpha$, $J_\beta$ and $%
J_{\alpha+\beta}$ are of non-complex type simultaneously. If $\Nij|_L = 0$,
then $\Nij_\Omega |_L = 0$ if and only if $\Omega (X_\alpha, X_\beta,
X_{-(\alpha+\beta)})=0$, where $L= L_\alpha \cup L_\beta \cup
L_{\alpha+\beta}$.
\end{teo}

\begin{proof}
Suppose $\Nij_\Omega |_L=0$. Since $\Nij|_L=0$, in particular we have
\begin{eqnarray*}
0 & = & \Nij_\Omega (x_\alpha S_\alpha+(a_\alpha - i)S^\ast _\alpha, x_\beta S_\beta+(a_\beta - i)S^\ast _\beta,x_{\alpha+\beta} S_{\alpha+\beta}+(a_{\alpha+\beta} - i)S^\ast _{\alpha+\beta}) \\
& = & \Omega (x_\alpha S_\alpha, x_\beta S_\beta , x_{\alpha+\beta} S_{\alpha+\beta})\\
& = & -2ix_\alpha x_\beta x_{\alpha+\beta} \Omega (X_\alpha, X_\beta, X_{-(\alpha+\beta)}).
\end{eqnarray*}

Reciprocally, suppose that $\Omega (X_\alpha,X_\beta,X_{-(\alpha+\beta)}) = 0$, thus we obtain
\begin{eqnarray*}
\Omega (A_\alpha , A_\beta ,S_{\alpha+\beta}) & = & \Omega (A_\alpha , S_\beta ,A_{\alpha+\beta})\\
& = & \Omega (S_\alpha , A_\beta ,A_{\alpha+\beta})\\
& = & \Omega (S_\alpha , S_\beta ,S_{\alpha+\beta})\\
& = & 0
\end{eqnarray*}
and the other cases were already zero, ensuring that $\Nij_\Omega |_L = 0$.
\end{proof}

\begin{cor}
Let $\mathcal{J}$ be an integrable generalized complex structure such that $%
J_\alpha$ is of non-complex type for all root $\alpha$ and consider $\Omega$
an invariant closed $3$-form. Then $\mathcal{J}$ is $\Omega$-integrable if
and only if $\Omega$ is zero.
\end{cor}

With the results proved until here we have a classification of the
integrable generalized complex structures $\mathcal{J}$ which are also $%
\Omega$-integrable. A natural question to be analyzed is the existence of
generalized almost complex structures which are non-integrable but which are 
$\Omega$-integrable.

Let $\mathcal{J}$ be a non-integrable generalized complex structure, then
there is a triple of roots $\alpha,\beta,\alpha+\beta$ such that $\Nij |_L
\not= 0$, where $L=L_\alpha\cup L_\beta \cup L_{\alpha+\beta}$. If $J_\alpha$%
, $J_\beta$ and $J_{\alpha+\beta}$ are not of non-complex type
simultaneously, is immediate to verify that if $\Nij|_L \not=0$, then $\Nij%
_\Omega |_L \not= 0$ for any invariant closed $3$-form $\Omega$. This way,
suppose $J_\alpha$, $J_\beta$ and $J_{\alpha+\beta}$ are of non-complex type
and suppose $\Nij|_L\not= 0$. With some computations, we get that $\Nij%
_\Omega |_L=0$ if and only if 
\begin{eqnarray*}
\frac{i}{6}m_{\alpha,\beta}(x_\alpha x_\beta(a_{\alpha+\beta} - i) -
x_\alpha x_{\alpha+\beta}(a_\beta - i) - x_\beta x_{\alpha+\beta} (a_\alpha
- i)) \\
- 2i x_\alpha x_\beta x_{\alpha+\beta} \Omega (X_\alpha, X_\beta,
X_{-(\alpha+\beta)}) = 0
\end{eqnarray*}
which is equivalent to 
\begin{eqnarray}  \label{omegaint}
\Omega (X_\alpha, X_\beta, X_{-(\alpha+\beta)}) = \frac{m_{\alpha,\beta}}{12}%
\left( \frac{a_{\alpha+\beta} - i}{x_{\alpha+\beta}} - \frac{a_\beta - i}{%
x_\beta} - \frac{a_\alpha - i }{x_\alpha} \right).
\end{eqnarray}

\begin{remark}
The expression (\ref{omegaint}) also cover the case in which $\Nij |_L=0$.
In fact, since $J_\alpha$, $J_\beta$ and $J_{\alpha+\beta}$ are of
non-complex type, then $\Nij|_L = 0$ implies that 
\begin{equation*}
x_{\alpha+\beta} = \frac{x_\alpha x_\beta}{x_\alpha + x_\beta} \ \mathrm{%
\ and } \ a_{\alpha+\beta} = \frac{a_\beta x_\alpha + a_\alpha x_\beta}{%
x_\alpha + x_\beta}.
\end{equation*}
Replacing these expressions of $x_{\alpha+\beta}$ and $a_{\alpha+\beta}$ in (%
\ref{omegaint}), we obtain 
\begin{equation*}
\Omega \left(X_\alpha, X_\beta, X_{-(\alpha+\beta)}\right)=0.
\end{equation*}
\end{remark}

Therefore, in general we have:

\begin{teo}
\label{teoomegaint} Let $\mathcal{J}$ be a generalized almost complex
structure and $\Omega$ an invariant closed $3$-form. Suppose there is a
triple of roots $\alpha,\beta,\alpha+\beta$ such that $J_\alpha$, $J_\beta$
and $J_{\alpha+\beta}$ are of non-complex type and denote $L=L_\alpha \cup
L_\beta \cup L_{\alpha+\beta}$. Then $\Nij_\Omega |_L = 0$ if and only if 
\begin{equation*}
\Omega (X_\alpha, X_\beta, X_{-(\alpha+\beta)}) = \frac{m_{\alpha,\beta}}{12}%
\left( \frac{a_{\alpha+\beta} - i}{x_{\alpha+\beta}} - \frac{a_\beta - i}{%
x_\beta} - \frac{a_\alpha - i }{x_\alpha} \right).
\end{equation*}
\end{teo}

To make it clearer that there are generalized almost complex structures
which are not integrable (in the usual sense), but are $\Omega$-integrable,
where $\Omega$ is an invariant closed $3$-form, let's do an example:

\begin{ex}
Consider the Lie algebra $A_3$, where we have just three positive roots. To
know, we are denote the simple roots by $\alpha_{12},\alpha_{23}$ and the
other positive root is $\alpha_{13} = \alpha_{12} + \alpha_{23}$. Consider a
generalized almost complex structure such that 
\begin{equation*}
J_{\alpha_{12}}=\left( 
\begin{array}{cccc}
1 & 0 & 0 & -1 \\ 
0 & 1 & 1 & 0 \\ 
0 & -2 & -1 & 0 \\ 
2 & 0 & 0 & -1%
\end{array}
\right), J_{\alpha_{23}}=\left( 
\begin{array}{cccc}
1 & 0 & 0 & -2 \\ 
0 & 1 & 2 & 0 \\ 
0 & -1 & -1 & 0 \\ 
1 & 0 & 0 & -1%
\end{array}
\right),
\end{equation*}
\begin{equation*}
J_{\alpha_{13}}=\left( 
\begin{array}{cccc}
1 & 0 & 0 & -1 \\ 
0 & 1 & 1 & 0 \\ 
0 & -2 & -1 & 0 \\ 
2 & 0 & 0 & -1%
\end{array}
\right)
\end{equation*}
that is, $x_{\alpha_{12}}=1$, $a_{\alpha_{12}} = 1$, $x_{\alpha_{23}}=2$, $%
a_{\alpha_{23}}=1$, $x_{\alpha_{13}}=1$ and $a_{\alpha_{13}} = 1$.

Observe that $\mathcal{J}$ is non-integrable. In fact, the $i$-eigenspace of 
$\mathcal{J}$ is $L = L_{\alpha_{12}}\cup L_{\alpha_{23}} \cup
L_{\alpha_{13}}$, and $\Nij|_L = 0$ when the system of the Theorem \ref%
{teo333} is satisfied. But, note that 
\begin{equation*}
x_{\alpha_{12}}x_{\alpha_{23}}-x_{\alpha_{12}}x_{\alpha_{13}}-x_{%
\alpha_{23}}x_{\alpha_{13}} = -1\not= 0.
\end{equation*}
Now, consider the invariant $2$-form $\omega$ defined by 
\begin{equation*}
\omega (X_{\alpha_{jk}},X_{-\alpha_{jk}}) = \frac{1}{12}\left( \frac{%
i-a_{\alpha_{jk}}}{x_{\alpha_{jk}}} \right)
\end{equation*}
for $1\leq j<k\leq 3$ and $0$ otherwise. Then $\Omega := d\omega$ is an
invariant closed $3$-form such that 
\begin{equation*}
\Omega (X_{\alpha_{12}}, X_{\alpha_{23}},X_{-\alpha_{13}}) = \frac{1}{12}(
\omega (X_{\alpha_{12}},X_{-\alpha_{12}}) + \omega (X_{\alpha_{23}},
X_{-\alpha_{23}}) + \omega (X_{-\alpha_{13}},X_{\alpha_{13}}) ).
\end{equation*}
Now, replacing the expression of $\omega$ we have 
\begin{equation*}
\Omega (X_{\alpha_{12}}, X_{\alpha_{23}},X_{-\alpha_{13}}) = \frac{i-1}{24}.
\end{equation*}
Defining $\Omega$ like this, by Theorem \ref{teoomegaint}, we have that $%
\mathcal{J}$ is $\Omega$-integrable.
\end{ex}

Let $\mathcal{J}$ be a non-integrable generalized complex structure and $%
\Omega$ an invariant closed $3$-form. With what was seen above, we have that
if a triple of roots $\alpha,\beta,\alpha+\beta$ such that $\Nij|_L\not= 0$
but $\Nij_\Omega |_L=0$, then we must necessarily have $J_\alpha$, $J_\beta$
and $J_{\alpha+\beta}$ of non-complex type. Therefore, analogous to Theorem %
\ref{existetheta} and Theorem \ref{dadotheta}, we have:

\begin{teo}
Let $\Omega$ be an invariant closed $3$-form and $\mathcal{J}$ be a $\Omega$%
-integrable generalized complex structure. Then there is a set $\Theta
\subset \Sigma$, where $\Sigma$ is a simple root system, such that $J_\alpha$
is of non-complex type for all $\alpha \in \langle \Theta \rangle$.
\end{teo}

\begin{teo}
Let $\Sigma $ be a simple root system and $\Theta \subset \Sigma $. Then
there exists a generalized almost complex structure $\mathcal{J}$ which is $%
\Omega $-integrable, for some invariant closed $3$-form $\Omega $, where $%
J_{\alpha }$ is of non-complex type for all $\alpha \in \langle \Theta
\rangle $.
\end{teo}

\end{document}